\pdfoutput=1
\documentclass{article}
\usepackage{graphicx,hyperref}
\usepackage{amsmath,amsfonts,amssymb,latexsym,amsthm}
\usepackage{multicol,multirow}

\usepackage{bm}
\providecommand{\B}{}
\renewcommand{\B}{\bm}
\newcommand{\D}{\partial}
\renewcommand{\le}{\leqslant}
\renewcommand{\ge}{\geqslant}
\newtheorem{theorem}{Theorem}[section]
\newtheorem{lemma}[theorem]{Lemma}
\newtheorem{proposition}[theorem]{Proposition}
\newtheorem{corollary}[theorem]{Corollary}
\newtheorem{assumption}{Assumption}
\theoremstyle{definition}

\newcommand{\newhbar}{\hbar}

\begin{document}



\title{On the order of accuracy of finite-­volume schemes on unstructured meshes}

\author{P.A.~Bakhvalov, M.D.~Surnachev}


\date{April 5, 2024}

\numberwithin{equation}{section}

\maketitle

\begin{abstract}
We consider finite-volume schemes for linear hyperbolic systems with constant coefficients on unstructured meshes. Under the stability assumption, they exhibit the convergence rate between $p$ and $p+1$ where $p$ is the order of the truncation error. Our goal is to explain this effect. The central point of our study is that the truncation error on $(p+1)$-th order polynomials has zero average over the mesh period. This condition is verified for schemes with a polynomial reconstruction,  multislope finite-volume methods, 1-exact edge-based schemes, and the flux correction method. We prove that this condition is necessary and, under additional assumptions, sufficient for the $(p+1)$-th order convergence. Furthermore, we apply the multislope method to a high-Reynolds number flow and explain its accuracy.
\end{abstract}

\medskip


\sloppy

\section{Introduction}

In this paper we study the convergence rate of finite-volume schemes for linear hyperbolic systems with constant coefficients on unstructured meshes. In particular, we consider the schemes with the polynomial reconstruction \cite{Barth1990,OllivierGooch1997,Titarev2007, Tsoutsanis2011, Antoniadis2017,Tsoutsanis2018,Farmakis2020,Tsoutsanis2021}, the cell-centered multislope method \cite{Buffard2010,Touse2015b}, and edge-based schemes  \cite{Dervieux1987,Barth1991,Barth1991b,Dervieux2000, Eliasson2008,Abalakin2016,Bakhvalov2017CAF,Bakhvalov2022b}, including the flux correction method \cite{Katz2012,Nishikawa2012,Pincock2014,Katz2015b,Thorne2016,Nishikawa2017}. For each scheme, we consider its linear (i. e. not monotonized) version. We do not specify a class of computational meshes directly. Instead, we assume that the scheme under consideration is $L_2$-stable on them.



If a scheme is $p$-exact on non-uniform meshes and $(p+1)$-exact on uniform meshes, then the convergence rate $p+1$ may happen on non-uniform meshes. This phenomenon is called supra-convergence \cite{Kreiss1986}. A simple example explaining this effect is presented in Appendix~\ref{sect:simple_example}. A similar effect is often observed on unstructured meshes \cite{Venkatakrishnan2003, Diskin2010, Diskin2011}.

In contrast to finite-element methods, linear high-order finite-volume methods have no solid mathematical background. Supra-convergence turns out to be a kind of miracle that sometimes happens and sometimes not -- but in both cases it is not clear why. Let us show an example.

Consider the Cauchy problem for the 1D transport equation \mbox{$w_t + w_x = 0$}, $w(0,x) = w_0(x)$. Take a non-uniform mesh with nodes $x_j$, $j \in \mathbb{Z}$, and denote $h_{j+1/2} = x_{j+1} - x_j$, $\newhbar_j = (x_{j+1}-x_{j-1})/2$. Let $D$ and $L$ be the operators taking each \mbox{$u = \{u_j, j \in \mathbb{Z}\}$} to $p_j'(x_j)$ and $p_j''(x_j)$, correspondingly, where $p_j(x)$ is the Lagrange interpolant of $u$ based on its values at $x_{j-1}$, $x_j$, $x_{j+1}$. The flux correction method (in terms of \cite{Nishikawa2017}, with the one-sided approximation of the source term) for this problem takes the form
\begin{equation}
\newhbar_j \frac{du_j}{dt} - \frac{h_{j+1/2}^3 + h_{j-1/2}^3}{24} \frac{d(Lu)_j}{dt} + u_j + \frac{1}{2} h_{j+1/2} (Du)_j - u_{j-1} - \frac{1}{2} h_{j-1/2} (Du)_{j-1} = 0,
\label{eq_intro_1}
\end{equation}
$u_j(0) = w_0(x_j)$. This scheme is 2-exact on non-uniform meshes (i. e. $u_j(t) = w_0(x_j-t)$ satisfies \eqref{eq_intro_1} for each quadratic polynomial $w_0$) and \mbox{3-exact} on uniform meshes. On non-uniform meshes, the flux correction method demonstrates the convergence with the third order (see  numerical results in \cite{Nishikawa2017}). However, consider a small modification of the flux correction method:
\begin{equation*}
\newhbar_j \frac{du_j}{dt} - \newhbar_j\frac{h_{j+1/2}^2 + h_{j-1/2}^2}{24} \frac{d(Lu)_j}{dt} + u_j + \frac{1}{2} h_{j+1/2} (Du)_j - u_{j-1} - \frac{1}{2} h_{j-1/2} (Du)_{j-1} = 0,
\end{equation*}
$u_j(0) = w_0(x_j)$. This scheme is also 2-exact on non-uniform meshes and 3-exact on uniform meshes. However, the numerical solution by this scheme converges with the second order only.

The development of numerical methods for hyperbolic systems on unstructured meshes is an ongoing process. The main instrument to predict the convergence rate is the truncation error analysis, and the supra-convergence is usually explained basing on the numerical results only (see, for instance, a recent paper \cite{Kong2023}). 

In this paper we propose a simple condition that helps to predict supra-convergence. Namely, we suggest to consider meshes that have a period. The truncation error on each $(p+1)$-th order polynomial should be zero {\it in average over the mesh period}. For finite-volume schemes, this condition is necessary for the scheme to be supra-convergent. Under additional assumptions it is also sufficient, i. e. yields the convergence with the order $p+1$.

The rest of the text is organized as follows. In Section~\ref{sect:refinement} we discuss strategies of mesh refinement, and Section~\ref{sect:history} contains a brief historical overview of the supra-convergence. In Section~\ref{sect:schemes} we describe the finite-volume schemes we study. The zero mean error condition is formulated in Section~\ref{sect:notation}. We verify this condition in a heuristic way in Section~\ref{sect:check1} and rigorously in Section~\ref{sect:check2}. In Section~\ref{sect:theory} we study the connection between the zero mean error condition and the $(p+1)$-th order of convergence. We also discuss why a scheme may satisfy the zero mean error condition while not possessing the convergence rate of order $p+1$. In Section~\ref{sect:numresults} we present numerical results.

\section{Mesh refinement}
\label{sect:refinement}

The concept of the order of accuracy implies that there is a set of meshes containing meshes with arbitrarily small step. The order of accuracy of a scheme may depend on this set. In this section we consider possible refinement strategies.

i. A set of meshes with assumptions on an element quality (minimal angle condition, maximal angle condition, etc.). A quasi-uniform assumption may be also applied. Such sets are frequently used in estimates for finite-element methods (see, for instance, \cite{Johnson1986}). However, in the multidimensional case, the $(p+1)$-th order convergence on these sets is a rare phenomenon. There are counter-examples for the basic finite-volume method \cite{Peterson1991} ($p=0$, convergence rate $1/2$) and two edge-based schemes \cite{Bakhvalov2023b} ($p=1$, convergence rates $5/4$ and $3/2$). For other finite-volume schemes, counter-examples may be constructed similarly.

ii. A set of meshes that are images of uniform meshes by a fixed smooth map. The ``smoothness'' of these meshes exhibits in the truncation error analysis, and this is not related to the supra-convergence.

iii. A set of meshes generated by the uniform refinement of a fixed mesh. On this set, the basic finite-volume scheme ($p=0$) converges with the first order \cite{Bouche2005}. There are no other results on this kind of mesh sets, up to the authors' best knowledge.

iv. A set of meshes generated by scaling a fixed mesh. Let an original mesh cover $\mathbb{R}^d$, $d \in \mathbb{N}$, and be periodic with respect to $x_i$, $i = 1, \ldots, d$. To refine this mesh, we scale the coordinates, see Fig.~\ref{fig:refinement}. On this kind of sets, supra-convergence happens for different kind of schemes. 

In this paper, we will not explicitly assume a strategy of the mesh refinement. However, the error estimate we prove in Section~\ref{sect:theory} depends on the constant $C_A$, which is hard to quantify. The mesh scaling does not change $C_A$, so our results are more meaningful if the meshes are refined by scaling.

\begin{figure}[t]
\centering
\includegraphics[width=\linewidth]{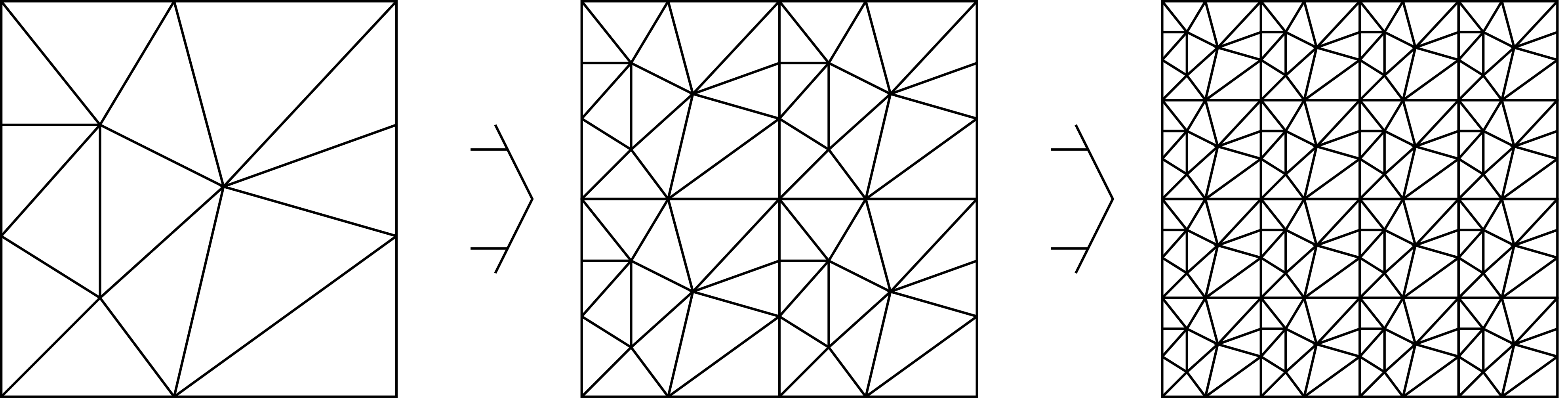}
\caption{Mesh refinement by scaling. Left to right: $h=1$, $h=1/2$, $h=1/4$}\label{fig:refinement}
\end{figure}

\section{A historical overview}
\label{sect:history}

The supra-convergence effect on non-uniform and unstructured meshes was studied or observed in many papers. Here we mention only those related to our study. In particular, we will not address finite-element methods (FEM) here. In FEM, the truncation error analysis is almost never used, and accuracy estimates are obtained using special techniques. Also, we will not consider papers where the supra-convergence was studied using the numerical results only.

i. The effect that the convergence rate may be higher than the order of the truncation error, was first discovered in \cite{Samarsky1962eng} for the boundary-value problem for the steady convection-diffusion equation. There was established an estimate of the solution error using a negative-norm estimate of the truncation error. By isolating a divegrence term in the truncation error, is was shown that the order of its negative norm is greater then the order of its integral norm. This explains the supra-convergence. Isolating a divergence term of a truncation error was used in many other papers, for instance, \cite{Despres2003}.


ii. The term ``supra-convergence'' was introduced in \cite{Kreiss1986}. In this paper, a scheme that is 3-exact on non-uniform meshes and 4-exact on unform meshes was considered. On non-uniform meshes, the solution error was $O(h^3)$. The explanation is that for a mesh with the periodic cell of three nodes, the truncation error on a 4-th order polynomial was not equal to zero in average.

iii. The scheme suggested for the Euler equations on triangular meshes in \cite{Jameson1986} is 1-exact for linear systems ($p=1$). It was found that the numerical solution converges with the second order. In \cite{Giles1989}, M.~Giles presented the following explanation. {\it If one considers a large control volume $\Omega$ whose area is $O(1)$, and sums over all of the cells inside $\Omega$ then each of the internal faces contributes equal and opposite amounts to the truncation errors in the two cells on either side. Hence the global integrated truncation error is $O(h^2)$.} Assume that the numerical dissipation effectively damps high-frequency oscillations with wavelength $O(h)$, i. e. they exist for a time $O(h)$. Then {\it the low-frequency component of the truncation error is second order and produces a solution error which is also second order. The high-frequency component of the truncation error is first order, but it produces a solution error which is also second order.}

Giles emphasized that his explanation is not a rigorous proof. However, if the mesh is periodic, his first argument is the zero mean error condition. As we prove below, together with his second assumption about the effective damping of short-frequency oscillations, it yields the second order convergence if the meshes are refined by scaling.

iv. In \cite{Bouche2005} the first order convergence for the basic finite-volume method was proved provided that the uniform refinement is used. The proof is based on the construction of an auxiliary mapping, which is the sum of the $L_2$-projection on the space of the piecewise-constant functions and the ``geometric corrector''. An explicit expression for the corrector was given. If we use a refinement by scaling instead of the uniform refinement, then the existence of the corrector for this scheme follows from the zero mean error condition.

v. The flux correction method was suggested in \cite{Katz2012} as a scheme for the steady Euler equations with a source term. It is 2-exact on unstructured simplicial meshes and gives the third order convergence. In \cite{Nishikawa2012}, it was extended to unsteady problems. The paper \cite{Nishikawa2012} is unique because the scheme was designed intentionally to preserve the supra-convergence effect. However, the reason why the supra-convergence happens in the steady flux correction method was still to be understood. In the subsequent papers on the flux correction method, supra-convergence was discussed basing on the numerical results only.

\section{Finite-volume schemes}
\label{sect:schemes}

\subsection{General form}

Consider the Cauchy problem for a linear hyperbolic system
\begin{equation}
\frac{\D w}{\D t} + \bm{A} \cdot \nabla w = 0, \quad \bm{r} \in \mathbb{R}^d, \quad t>0, \label{eq001}
\end{equation}
\begin{equation}
w(0,\bm{r}) = w_0(\bm{r}), \quad \bm{r} \in \mathbb{R}^d. \label{eq002}
\end{equation}
Here $d \in \mathbb{N}$ is the space dimension, $\bm{A} = (A_1, \ldots, A_d)$ is a set of $(n \times n)$-matrices with constant coefficients, and $w_0 \in (C^1(\mathbb{R}^d))^n$ is 1-periodic with respect to each coordinate axis.

Let $\mathbb{R}^d$ be covered by a conform mesh that is invariant with respect to the translation in each vector with integer components. Let $\mathcal{N}$ be the set of mesh nodes and $\mathcal{E}$ be the set of mesh elements. For a node $j \in \mathcal{N}$, denote by $\bm{r}_j$ its radius-vector. For an element $j \in \mathcal{E}$, denote by $\bm{r}_j$ the radius-vector of its mass center.

The concept of finite-volume methods implies that the computational domain (in our case, $\mathbb{R}^d$) is represented as a union of non-overlapping control volumes. Let $\mathcal{M}$ be the set of control volumes. For cell-centered schemes, a control volume coincides with a mesh element, i. e. $\mathcal{M} \equiv \mathcal{E}$. For vertex-centered schemes, each control volume corresponds to a vertex, i. e. $\mathcal{M} \equiv \mathcal{N}$. A mesh function $f \in \mathbb{C}^{\mathcal{M}}$ is a sequence $f = \{f_j \in \mathbb{C}, j \in \mathcal{M}\}$. 

For $j \in \mathcal{M}$, let $K_j \subset \mathbb{R}^d$ be the shape of the control volume, $|K_j|$ be its measure. For $j \in \mathcal{M}$, let $N(j)$ be the set of $k \in \mathcal{M} \setminus \{j\}$ such that $\D K_j \cap \D K_k$ has nonzero $(d-1)$-dimensional measure. For $j \in \mathcal{M}$ and $k \in N(j)$ put
\begin{equation}
\bm{n}_{jk} = \int\limits_{\D K_j \cap \D K_k} \bm{n} dS
\label{eq_def_njk}
\end{equation}
where $\bm{n}$ is the unit normal oriented inside $K_k$. Obviously, $\bm{n}_{jk} = -\bm{n}_{kj}$.

Most of finite-volume schemes for \eqref{eq001} can be represented in the general form
\begin{equation}
\frac{du_j}{dt} + \frac{1}{|K_j|} \sum\limits_{k \in N(j)} F_{jk}[u] = 0, \quad j \in \mathcal{M},
\label{eqFVgen0}
\end{equation}
where the numerical fluxes satisfy $F_{jk}[u] = -F_{kj}[u]$. We assume that $F_{jk}$ are linear, depend on a finite number of $u_l$, and are invariant with respect to the mesh translation. For the flux correction method, a special approximation of the unsteady term is used, see below. Upwind fluxes are commonly used:
\begin{equation}
F_{jk}[u] =\frac{\bm{A}\cdot\B{n}_{jk} + |\bm{A}\cdot\B{n}_{jk}|}{2} \mathcal{R}_{jk}[u] + \frac{\bm{A}\cdot\B{n}_{jk} - |\bm{A}\cdot\B{n}_{jk}|}{2} \mathcal{R}_{kj}[u]. \label{upwind}
\end{equation}
Here $|\bm{A}\cdot\B{n}_{jk}| = S_{jk} |\Lambda_{jk}| S_{jk}^{-1}$ where $S_{jk}$ and $\Lambda_{jk}$ are the matrices of eigenvectors and eigenvalues of $\bm{A}\cdot\B{n}_{jk}$. 

Operators $\mathcal{R}_{jk}$ define a specific finite-volume scheme. For instance, taking $\mathcal{R}_{jk}[u] = u_j$ we get the basic finite-volume method, which coincides with the discontinuous Galerkin method based on zero-order polynomials.

The system \eqref{eqFVgen0} is accompanied by the initial data:
$$
u(0) = \Pi w_0
$$
where $\Pi$ is a linear map of $(C(\mathbb{R}^d))^n$ to $(\mathbb{C}^{\mathcal{M}})^n$.

\subsection{Schemes with a polynomial reconstruction}

Here a control volume may be either a mesh element or a dual cell. In the latter case, a method to construct a dual cell does not matter. We assume that the faces of the control volumes are planar, otherwise \eqref{upwind} should be replaced by a more complicated expression.

Let $\Pi$ be the operator taking a continuous function $f$ to its average over control volumes. Let $p \in \mathbb{N} \cup \{0\}$ and $P_p$ be the space of $p$-th order polynomials of $d$ variables. A finite-volume scheme with the $p$-th order polynomial reconstruction has the form \eqref{eqFVgen0}, \eqref{upwind}. The operators $\mathcal{R}_{jk}$ are defined as follows.

For each $j \in \mathcal{M}$ denote by $S_j \subset \mathcal{M}$ the stencil of reconstruction, $\dim P_p \le |S_j| < \infty$. Assume that $S_j$ is invariant with respect to any translation and scaling of the mesh.

For $n>1$, the operator $\mathcal{R}_{jk}[u]$ acts componentwise, so it is enough to define it for $n=1$. Consider the following system
\begin{equation}
\frac{1}{|K_k|}\int\limits_{K_k} p_j(\bm{r}) dv = u_k, \quad k \in S_j,
\label{eq_polynomial_FV}
\end{equation}
with respect to the coefficients of polynomial $p_j \in P_p$. If $|S_j| > \dim P_p$, then this system is overdetermined. Its solution is usually defined as the polynomial $p_j$ that gives the minimum of
$$
\sum\limits_{k \in S_j} w_{jk} \biggl(u_k - \frac{1}{|K_k|}\int\limits_{K_k} p_j(\bm{r}) dv\biggr)^2
$$
with some $w_{jk}>0$ while satisfying \eqref{eq_polynomial_FV} for $k = j$. Finally, put
$$
\mathcal{R}_{jk}[u] = \int\limits_{\D K_j \cap \D K_k} p_{j}(\bm{r}) dS.
$$


\subsection{Multislope cell-centered scheme}
\label{sect:def_multislope}


For simplicity, assume that each mesh face is planar. Consider the scheme \eqref{eqFVgen0}, \eqref{upwind} with the following definition of $\mathcal{R}_{jk}[u]$. For $j \in \mathcal{E}$ and $k \in N(j)$, let $\bm{r}_{jk}$ be the radius-vector of the mass center of face $\D K_j \cap \D K_k$. For $j \in \mathcal{E}$ let $A(j) \subset \mathcal{E}$ be the set of elements having at least one common vertex with $j$. 

In 2D, consider the ray originating at $\bm{r}_{jk}$ and passing through $\bm{r}_j$. Let $\bm{r}_{jk}^-$ be the intersection of this ray with a segment of the form $[\bm{r}_m, \bm{r}_n]$, $m,n \in A(j)$. If there is more than one intersection, use the one furtherst from $\bm{r}_j$, see Fig.~\ref{fig:bbr}.

\begin{figure}[t]
\centering
\includegraphics[width=0.7\linewidth]{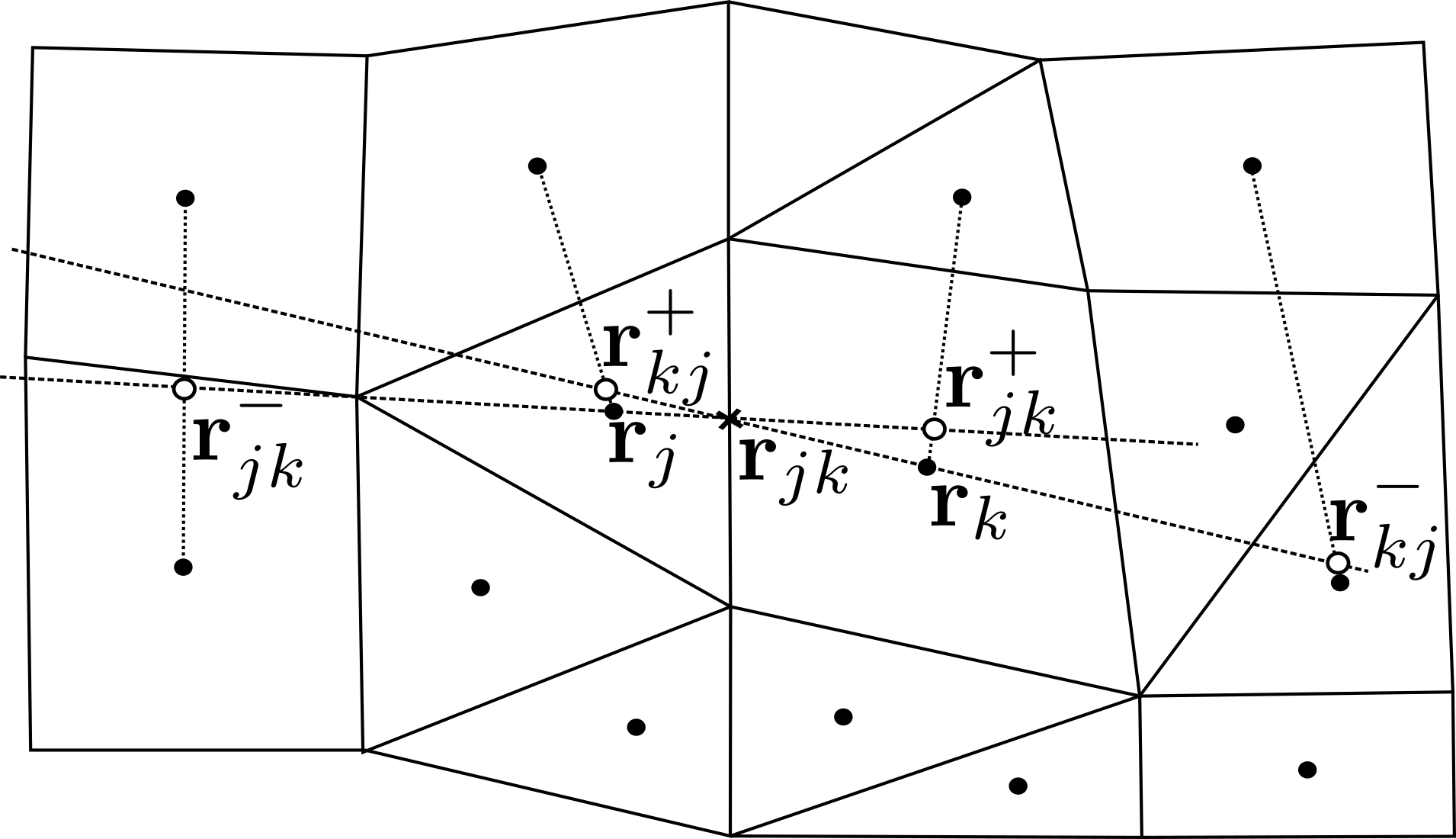}
\caption{Construction of the numerical flux in the multislope scheme}\label{fig:bbr}
\end{figure}
\begin{figure}[t]
\centering
\includegraphics[width=0.8\linewidth]{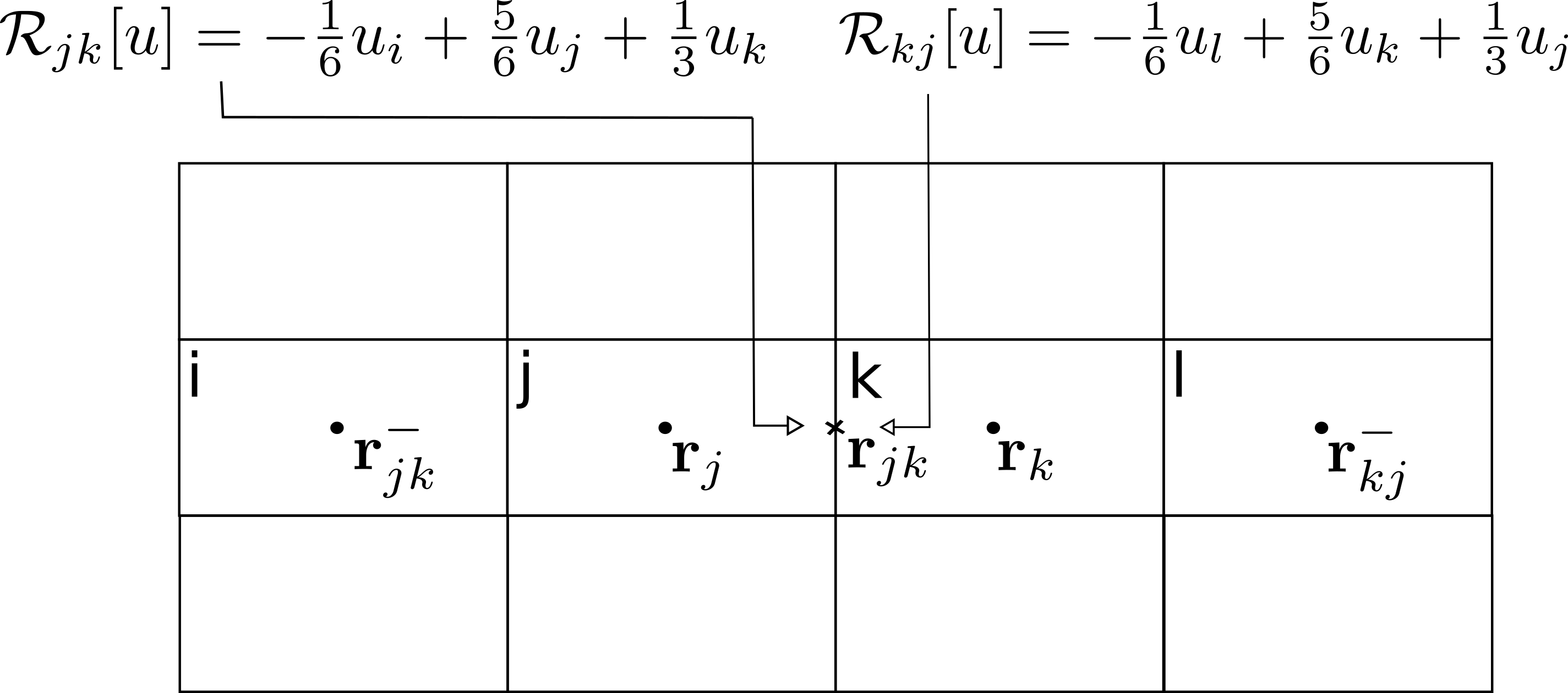}
\caption{Construction of the numerical flux in the multislope scheme on a uniform Cartesian mesh}\label{fig:bbr:cart}
\end{figure}

Consider the ray originating at $\bm{r}_{jk}$ and going in the opposite direction. Let $\bm{r}_{jk}^+$ be the intersection of this ray with a segment of the form $[\bm{r}_k, \bm{r}_m]$, $m \in A(j)$. If there is more than one intersection, use the one furtherst from $\bm{r}_j$.

If $d>2$, the algorithm is the same, but the segments of the form $[\bm{r}_m, \bm{r}_n]$ are replaced by the $(d-1)$-simplexes with vertices from $A(j)$.

Let $u_{jk}^+$ and $u_{jk}^-$ be defined by a linear interpolation by these simplexes to $\bm{r}_{jk}^+$ and $\bm{r}_{jk}^-$, correspondingly. Put
\begin{equation}
\mathcal{R}_{jk}[u] = u_j + |\bm{r}_{jk} - \bm{r}_j| \left(\frac{2}{3} \frac{u_{jk}^+ - u_j}{|\bm{r}_{jk}^+ - \bm{r}_j|} + \frac{1}{3} \frac{u_j - u_{jk}^-}{|\bm{r}_j - \bm{r}_{jk}^-|}\right).
\label{eq_R_BBR3}
\end{equation}
If at least one of the points $\bm{r}_{jk}^{\pm}$ is not defined, put  $\mathcal{R}_{jk}[u]=u_j$. We call this scheme BBR3. The use of a flux limiter instead of the constant weights ($1/3$ and $2/3$) yields the Multislope MUSCL Method \cite{Touse2015b}. Also in \cite{Touse2015b} a more efficient procedure to define $\bm{r}_{jk}^{\pm}$ is suggested.

Let $\Pi$ be the pointwise mapping, i. e. the operator taking each continuous function $f$ to its values at element centers. The scheme BBR3 is 1-exact as soon as for each $j \in \mathcal{N}$ and $k \in N(j)$ the points $\bm{r}_{jk}^{\pm}$ are defined. On Cartesian meshes, BBR3 reduces to the 3-th order finite difference scheme, see Fig.~\ref{fig:bbr:cart}.

On a regular-triangular mesh, BBR3 is 2-exact. We prove this fact in Section~\ref{sect:rig:BBR3} and show that this leads to the 3-rd order convergence for steady problems (if $\bm{A} \cdot \nabla w_0 = 0$).

\subsection{1-exact edge-based schemes}
\label{sect:EB_def}

Now we move to the class of edge-based FV schemes. Mesh functions are defined at nodes, i. e. $\mathcal{M} = \mathcal{N}$. $\Pi$ is the pointwise mapping, i. e. the operator taking each continuous function $f$ to its values at mesh nodes. 

In contrast to standard FV schemes, here a method to construct the dual cells is essential to guarantee the 1-exactness. We describe it for simplicial mesh. A generalization to mixed-element mesh may be done by the semi-transparent control volumes, see Section IV in \cite{Bakhvalov2016AIAA}.

Each point $\bm{r} \in \mathbb{R}^d$ lays inside a simplex $e \in \mathcal{E}$ or on its boundary. Then $\bm{r}$ is a convex combination of radius-vectors of the vertices of $e$. By definition, the point $\bm{r}$ belongs to the control volume of any vertex with the largest coefficient in this combination. This represents $\mathbb{R}^d$ as a union of non-overlapping control volumes $K_j$, up to the set of zero measure. For instance, in 2D, medians of a triangle split it into six parts, and $K_j$ is a union of all such parts containing $\bm{r}_j$, see Fig.~\ref{fig:cells2d}. Then $N(j)$ is the set of nodes connected to $j$ by edge. Obviously,
\begin{equation}
|K_j| = \frac{1}{d+1}\sum\limits_{e \ni j} |e|
\end{equation}
where $|e|$ is the volume of simplex $e$.

\begin{figure}[t]
\centering
\includegraphics[width=0.3\linewidth]{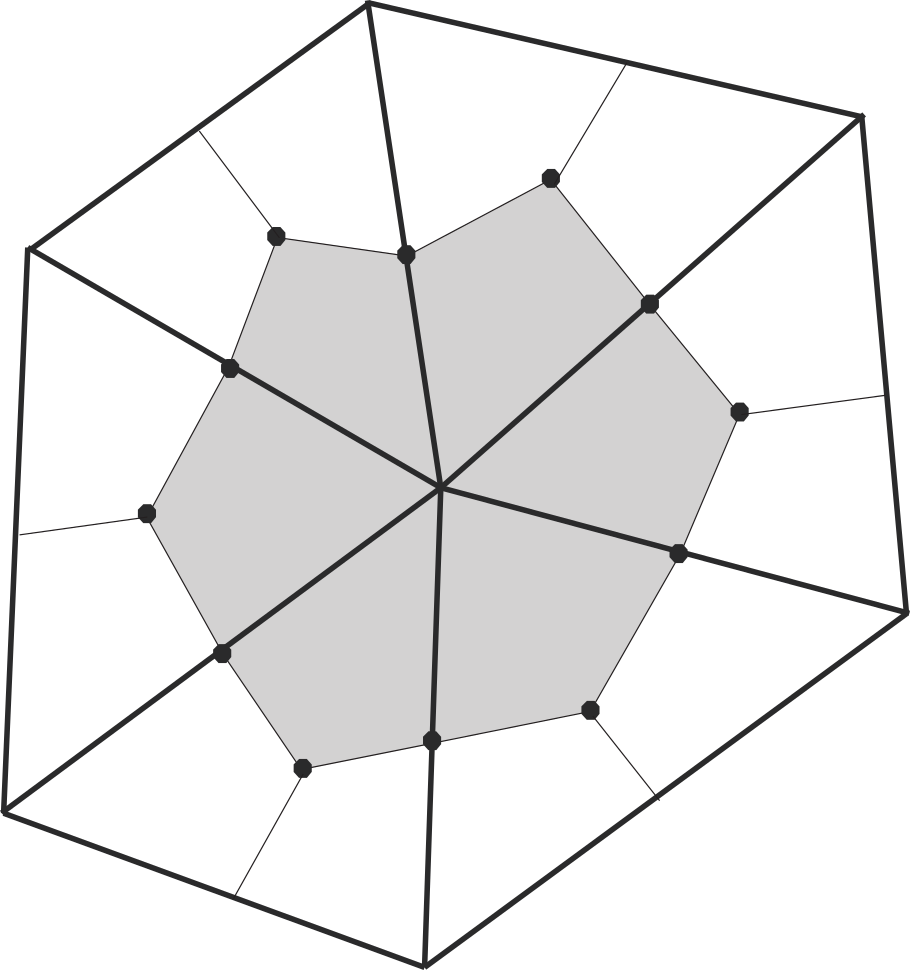}
\caption{Median cell on a 2D triangular mesh}\label{fig:cells2d}
\end{figure}

In this subsection, we consider only schemes of the form \eqref{eqFVgen0}. A scheme of the form \eqref{eqFVgen0} is called edge-based if the numerical fluxes are 1-exact at edge centers, i.~e. for each $f \in (P_1)^n$, each $j \in \mathcal{N}$, and each $k \in N(j)$ there holds
\begin{equation}
F_{jk}[\Pi f] = \bm{A} \cdot \bm{n}_{jk}\ f\left(\frac{\bm{r}_j + \bm{r}_k}{2}\right). \label{EB}
\end{equation}
An example of an edge-based scheme is the mass-lumped standard Galerkin method, which has the form \eqref{eqFVgen0} with $F_{jk}[u] = \bm{A} \cdot \bm{n}_{jk}\,(u_j+u_k)/2$ (see Appendix~\ref{sect:barycentric}). Since the standard Galerkin method is 1-exact on any simplicial mesh, this extends to each edge-based scheme. Formally speaking,
\begin{equation}
\frac{1}{|K_j|} \sum\limits_{k \in N(j)} \bm{n}_{jk} \otimes \frac{\bm{r}_k + \bm{r}_j}{2} = I
\label{eq_1_exact}
\end{equation}
where $I$ is the unit tensor.

For the purpose of this research, more details of the numerical flux definition are unnecessary.

\subsection{Flux correction method}

Here we consider the steady flux correction (FC) method with three options for the discretization of the unsteady term: the pointwise approximation (leading to the steady FC method), the ``divergence formulation'' \cite{Nishikawa2012}  and ``extended Galerkin formula'' \cite{Pincock2014,Katz2015b,Thorne2016}. Like in other edge-based schemes, $\mathcal{M} = \mathcal{N}$ and $\Pi$ is the pointwise mapping. 

The {\it steady flux correction method} has the form \eqref{eqFVgen0}, \eqref{upwind} with
\begin{equation}
\mathcal{R}_{jk}[u] = u_j + \frac{\bm{r}_k - \bm{r}_j}{2} \cdot (\bm{G} u)_j
\label{eq_R_FC}
\end{equation}
where $\bm{G}$ is a 2-exact approximation of the gradient (i. e. $\bm{G} \Pi f = \Pi \nabla f$ for each $f \in (P_2)^n$). The steady flux correction method is 2-exact on any simplicial unstructured mesh \cite{Katz2012}. It usually demonstrates the third order convergence on problems with steady solutions.

Now consider the flux correction method with the {\it ``divergence formulation''} of the unsteady term. For brevity, we restrict our analysis to the 2D case. By $x_l$, $y_l$ we denote the components of $\bm{r}_l$, $l \in \mathcal{N}$. For $j \in \mathcal{N}$ and $f \in \mathbb{C}^{\mathcal{M}}$, define  $\bm{v}_{j}[f] \in (\mathbb{C}^{\mathcal{M}})^2$ with the components
\begin{equation}
(\bm{v}_j[f])_l = \left(\begin{array}{c}
(x_l-x_j)f_l/2 - (x_l-x_j)^2 (D_x f)_l/4 + (x_l-x_j)^3 (D_{xx} f)_l/12 \\
(y_l-y_j)f_l/2 - (y_l-y_j)^2 (D_y f)_l/4 + (y_l-y_j)^3 (D_{yy} f)_l/12
\end{array}\right)
\label{eq_def_vj}
\end{equation}
where $D_x$, $D_y$, $D_{xx}$, $D_{yy}$ are 2-exact approximations of the corresponding derivatives. This guarantees that for each for $g \in P_2$ there holds 
\begin{equation}
(\bm{v}_j[\Pi g])_l = (\bm{w}[g])_l - (\bm{w}[g])_j
\label{eq_def_wl0}
\end{equation}
where $\bm{w}[g]$ has the components
\begin{equation}
(\bm{w}[g])_l = \frac{1}{2}\left(\int\limits_0^{x_l} g(x,y_l) dx, 
\int\limits_0^{y_l} g(x_l, y) dy\right)^T.
\label{eq_def_wl1}
\end{equation}

Now for $j \in \mathcal{N}$ approximate the divergence of $\bm{v}_j[f]$ at the point $\bm{r}_j$ as
$$
(\mathrm{DIV}\,\bm{v}_j[f])_j = \frac{1}{|K_j|} \sum\limits_{k \in N(j)} \phi_{jk}[\bm{v}_j[f]], \quad
\phi_{jk}[\bm{v}_j[f]] = \bm{n}_{jk} \cdot \frac{\mathcal{R}_{jk}[\bm{v}_j[f]] + \mathcal{R}_{kj}[\bm{v}_j[f]]}{2}
$$
with $\mathcal{R}_{jk}$ and $\mathcal{R}_{kj}$ defined by \eqref{eq_R_FC}.
Substituting the expression \eqref{eq_def_vj} for $\bm{v}_j$ we obtain an approximation of the form
$$
(\mathrm{DIV}\,\bm{v}_j[f])_j = \sum\limits_{k \in \tilde{N}(j)} m_{jk}f_k
$$
where $m_{jk}$ are some scalar coefficients and $\tilde{N}(j)$ are some finite sets. Now use these coefficients $m_{jk}$  to write the final approximation
\begin{equation}
\sum\limits_{k \in \tilde{N}(j)} m_{jk} \frac{du_k}{dt} + \frac{1}{|K_j|} \sum\limits_{k \in N(j)} F_{jk}[u] = 0.
\label{eqFCgen0}
\end{equation}
with $F_{jk}$ given by \eqref{upwind}, \eqref{eq_R_FC}.

The flux correction method with the use of the  {\it ``extended Galerkin formula''} has the form
\begin{equation}
\sum\limits_{k \in N(j)}  v_{jk} \, s_{jk}\!\!\left[\frac{du}{dt}\right] + \sum\limits_{k \in N(j)} F_{jk}[u] = 0,
\label{eqFC}
\end{equation}
\begin{equation}
v_{jk} = \frac{1}{2d}(\bm{r}_k - \bm{r}_j) \cdot \bm{n}_{jk},
\label{eq_def_vjk}
\end{equation}
where $F_{jk}$ are given by \eqref{upwind}, \eqref{eq_R_FC}, and the functionals 
$s_{jk}$ satisfy
\begin{equation}
s_{jk}[\Pi p] = p(\bm{r}_j) - \frac{d}{4(d+2)} ((\bm{r}_k - \bm{r}_j) \cdot \nabla)^2 p \quad \mathrm{for\ each} \quad p \in (P_2)^n.
\label{eq_def_sjk}
\end{equation}
The condition \eqref{eq_def_sjk} is equivalent to (5.20) in \cite{Nishikawa2017}. We assume that $s_{jk}$ are invariant with respect to a mesh translation.

\section{Zero mean error condition}
\label{sect:notation}

In this section we introduce the zero mean error condition. In the two following sections, we check it for the finite-volume schemes defined in Section~\ref{sect:schemes}.

Recall that we consider meshes that are invariant under translations by vectors with integer components. We call nodes and elements equivalent if one of them is the image of another under a translation of this type. Let $\mathcal{N}^1$ and $\mathcal{E}^1$ be the sets of nodes and elements, correspondingly, on the unit cube (or, that is the same, the sets of nodes and elements where equivalent ones are identified). Let $\mathcal{M}^1$ be the set of degrees of freedom on the unit cube. Depending on the scheme, $\mathcal{M}^1 = \mathcal{N}^1$ or $\mathcal{M}^1 = \mathcal{E}^1$. Then $\mathcal{M} = \mathbb{Z}^d \times \mathcal{M}^1$.
By construction,
$$
\sum\limits_{j \in \mathcal{M}^1} |K_j| = 1.
$$

Let $\Pi$ be a linear map that takes each $f \in (C(\mathbb{R}^d))^n$ either to its point values at $\bm{r}_j$, $j \in \mathcal{M}$, or to its integral averages over $K_j$, $j \in \mathcal{M}$. Consider a semidiscrete scheme of the form
\begin{equation}
\sum\limits_{k \in \mathcal{S}(j)} m_{jk} \frac{du_k}{dt} + \sum\limits_{k \in \mathcal{S}(j)} a_{jk} u_k = 0, \quad j \in \mathcal{M},
\label{eq003}
\end{equation}
\begin{equation}
u_j(0) = (\Pi w_0)_j, \quad j \in \mathcal{M}.
\label{eq004}
\end{equation}
Here $u_k(t) \in \mathbb{C}^n$, the scheme stencil $\mathcal{S}(j) \subset \mathcal{M}$ for each $j \in \mathcal{M}$ is a finite set, and $m_{jk}$ and $a_{jk}$ are real-valued $(n \times n)$-matrices. We assume that a solution of \eqref{eq003}--\eqref{eq004} satisfying \mbox{$\sup\limits_{0<t<T} \sup\limits_{j \in \mathcal{M}} \|u_j(t)\| < \infty$} for each $T>0$ exists and is unique. Here and below the Euclidean norm on $\mathbb{C}^n$ is used.

Assume the normalization condition 
$$
\sum\limits_{k \in \mathcal{S}(j)} m_{jk} = I
$$
where $I$ is the unit $(n \times n)$-matrix. The truncation error of \eqref{eq003}--\eqref{eq004} on $f \in (C^{1}(\mathbb{R}^d))^n$ in the sense of $\Pi$ is the set $\epsilon(f, \Pi) =\{\epsilon_j(f, \Pi) \in \mathbb{C}^n,\ j\in \mathcal{M}\}$ defined by
\begin{equation}
\epsilon_j(f, \Pi) = 
- \sum\limits_{k \in \mathcal{S}(j)} m_{jk} (\Pi (\bm{A} \cdot \nabla f))_k + \sum\limits_{k \in \mathcal{S}(j)} a_{jk} (\Pi f)_k.
\label{def00_apprerr}
\end{equation}

Let $P_p$ be the space of $p$-th order polynomials of $d$ variables. We call the scheme \eqref{eq003}--\eqref{eq004} $p$-exact iff $\epsilon_j(f,\Pi) = 0$ holds for each $f \in (P_p)^n$  and each \mbox{$j \in \mathcal{M}$}.

Let the scheme \eqref{eq003}--\eqref{eq004} be $p$-exact. The mean truncation error on $f \in (P_{p+1})^n$ is
\begin{equation}
\bar{\epsilon}(f,\Pi) = \sum\limits_{j \in \mathcal{M}^1} |K_j|\ \epsilon_j(f,\Pi).
\label{def00_zeroav_0}
\end{equation}
We say that the scheme has {\it zero mean error} on the $(p+1)$-th order polynomials iff $\bar{\epsilon}(f,\Pi) = 0$ holds for each $f \in (P_{p+1})^n$, i. e.
\begin{equation}
\sum\limits_{j \in \mathcal{M}^1} |K_j|\ \epsilon_j(f,\Pi) = 0 \quad \mathrm{for\ each} \quad f \in (P_{p+1})^n.
\label{def00_zeroav}
\end{equation}


\section{Mean truncation error. Heuristic approach}
\label{sect:check1}

In this section we study the mean truncation error for the finite-volume methods defined in Section~\ref{sect:schemes}. Assuming the $(p+1)$-exactness on uniform meshes we {\it explain} why the zero mean error property  extends from uniform meshes to unstructured meshes. A {\it formal} proof of the zero mean error condition for the schemes defined in Section~\ref{sect:schemes} is presented in Section~\ref{sect:check2}.

What is a uniform mesh, depends on the specific scheme. For our schemes it is either a uniform cubic mesh or a simplicial TI-mesh. A mesh is translationally-invariant (TI) if it is invariant with respect to a translation by each mesh edge. In 2D, a TI-mesh is either a uniform quad mesh, or a regular-triangular mesh, or their image by a linear transform. 


For $m \in \mathbb{N}$ by $\mathcal{M}^m = \{0, \ldots, m-1\}^d \times \mathcal{M}^1$ denote the set of degrees of freedom on the cube with edge length $m$. Denote
$$
V^m = \bigcup\limits_{j \in \mathcal{M}^m} K_j.
$$

\subsection{Schemes with a polynomial reconstruction}

We begin with finite-volume methods where $\Pi$ is the operator taking a continuous function $f$ to its average over control volumes.

Assume that the scheme is $p$-exact on an arbitrary mesh and $(p+1)$-exact on uniform cubic meshes. Consider the mean truncation error on a function $f \in (P_{p+1})^n$. By $p$-exactness, $\epsilon_j(f,\Pi) = \epsilon_{\tilde{j}}(f,\Pi)$ if $j \in \mathcal{M}$ and $\tilde{j} \in \mathcal{M}$ are equivalent. Then for each $m \in \mathbb{N}$ we have
\begin{equation*}
\begin{gathered}
m^d \bar{\epsilon}(f,\Pi) = \sum\limits_{j \in \mathcal{M}^m} |K_j| \epsilon_j(f,\Pi) =
\sum\limits_{j \in \mathcal{M}^m} \left(-\int\limits_{K_j} (\bm{A}\cdot \nabla) f dv +  \sum\limits_{k \in N(j)} F_{jk}[\Pi f] \right) =
\\
=
- \oint\limits_{\D V^m} \bm{A} f   \cdot \bm{n} dS
+ \sum\limits_{\langle j, k \rangle\ :\ j \in \mathcal{M}^m, k \in N(j) \setminus \mathcal{M}^m} F_{jk}[\Pi f].
\end{gathered}
\end{equation*}
The last sum is over the faces of control volumes sharing a control volume from $\mathcal{M}^m$ and a control volume from $\mathcal{M} \setminus \mathcal{M}^m$. Hence, $m^d \bar{\epsilon}(f,\Pi)$ depends only on the mesh in a strip with width $O(1)$ as $m \rightarrow \infty$. 

Consider an auxiliary mesh with period $m$ constructed as follows. Keep the strip of the width $O(1)$ as is. Outside a wider strip, also of the width $O(1)$, use a uniform mesh, see Fig.~\ref{fig:meshtrick}. By assumption, the scheme is $p$-exact on this mesh. Assume that the truncation error on $(p+1)$-order polynomials at each $j \in \mathcal{M}$ is $O(1)$ as $m \rightarrow \infty$. On the new mesh, most of the terms in the sum
$$
\sum\limits_{j \in \mathcal{M}^m} |K_j| \epsilon_j(f,\Pi)
$$
are zero, and the sum of $|K_j|$ corresponding to the nonzero truncation error is $O(m^{d-1})$ as $m \rightarrow \infty$. Thus, we come to the inequality
$$
m^d \bar{\epsilon}(f,\Pi)  = O(m^{d-1}) \quad \mathrm{as} \quad m \rightarrow \infty,
$$
which proves $\bar{\epsilon}(f,\Pi) = 0$.

\begin{figure}[t]
\centering
\includegraphics[width=\linewidth]{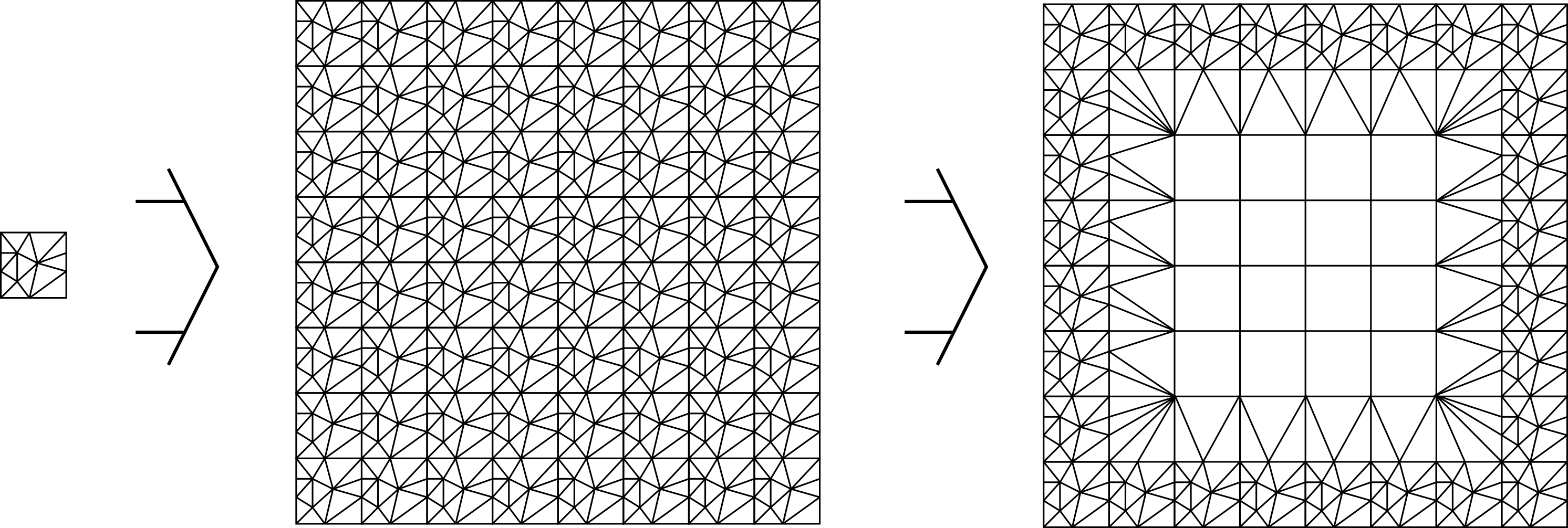}
\caption{The use of an auxiliary mesh to prove the zero-mean-error property. Left: mesh in the unit cube; middle: mesh in the cube with edge $m$; right: auxiliary mesh}\label{fig:meshtrick}
\end{figure}

\subsection{Multislope cell-centered scheme}

The multislope cell-centered scheme defined in Section~\ref{sect:def_multislope} is 1-exact on unstructured meshes, 2-exact on regular triangular meshes, and 3-exact on uniform cubic meshes. We are not able to link a regular-triangular mesh to a uniform cubic mesh while preserving the 2-exactness. So the auxiliary mesh argument fails here. We return to this scheme in the Section~\ref{sect:check2}.

\subsection{1-exact edge-based schemes}
\label{sect:EB_heu}

Now we move to the class of edge-based schemes. Mesh functions are defined at nodes, i. e. $\mathcal{M} = \mathcal{N}$. $\Pi$~is the pointwise mapping, i. e. the operator taking each continuous function $f$ to its values at mesh nodes. 

For each $m \in \mathbb{N}$ and $f \in (P_2)^n$ we have
\begin{equation*}
\begin{gathered}
m^d \bar{\epsilon}(f,\Pi) = \sum\limits_{j \in \mathcal{M}^m} |K_j|\ \epsilon_j(f,\Pi) =
\sum\limits_{j \in \mathcal{M}^m} \left(-|K_j|\ \bm{A}\cdot \nabla f(\bm{r}_j) +  \sum\limits_{k \in N(j)} F_{jk}[\Pi f] \right) =
\\
=
- \sum\limits_{j \in \mathcal{M}^m} |K_j|\ \bm{A} \cdot \nabla f(\bm{r}_j)
+ \sum\limits_{\langle j, k \rangle\ :\ j \in \mathcal{M}^m, k \in N(j)\setminus \mathcal{M}^m} F_{jk}[\Pi f].
\end{gathered}
\end{equation*}

Let $g(\bm{r})$, $\bm{r} \in V^m$, be the piecewise-constant function with the values $g(\bm{r}) = \bm{A} \cdot \nabla f(\bm{r}_j)$ for $\bm{r} \in K_j$. Let $\tilde{\mathcal{E}}^m$ be set of simplexes that lay inside $V^m$ and $\tilde{V}^m$ be their union. Then
$$
\sum\limits_{j \in \mathcal{M}^m} |K_j|\ \bm{A} \cdot \nabla f(\bm{r}_j) = \int\limits_{V^m} g(\bm{r}) dV = 
\int\limits_{V^m \setminus \tilde{V}^m} g(\bm{r}) dV + \sum\limits_{e \in \tilde{\mathcal{E}}^m} \int\limits_{e} g(\bm{r}) dV.
$$
Each simplex $e \in \tilde{\mathcal{E}}^m$ consists of $d+1$ parts with equal volumes, each of them belonging to a corresponding vertex of $e$. Since $\bm{A} \cdot \nabla f \in (P_1)^n$, then for each simplex $e$ there holds $\int_{e} g(\bm{r}) dV = \int_e \bm{A} \cdot \nabla f(\bm{r}) dV$. Hence,
\begin{equation*}
\begin{gathered}
\sum\limits_{j \in \mathcal{M}^m} |K_j|\ \bm{A} \cdot \nabla f(\bm{r}_j) = 
\int\limits_{V^m \setminus \tilde{V}^m} g(\bm{r}) dV + 
\int\limits_{\tilde{V}^m} \bm{A} \cdot \nabla f(\bm{r})dV
=
\\
= \int\limits_{V^m \setminus \tilde{V}^m} g(\bm{r}) dV + 
\int\limits_{\D\tilde{V}^m} f(\bm{r}) \bm{A}\cdot\B{n} dS.
\end{gathered}
\end{equation*}
From here it is obvious that $m^d \bar{\epsilon}(f,\Pi)$ depends only on the mesh in a strip with width $O(1)$ near the boundary of $V^m$.

The rest of the analysis repeats the previous case. The only difference is that a uniform mesh here is a simplicial TI-mesh. The extension to mixed-element meshes is straightforward, provided that the semi-transparent control volumes \cite{Bakhvalov2016AIAA} are used.

\subsection{Flux correction method} 

Like in other edge-based schemes, the mesh functions are defined at nodes, and $\Pi$ is the pointwise mapping. The flux correction method is 2-exact \cite{Katz2012}.

First consider the {\it steady flux correction method} \eqref{eqFVgen0}, \eqref{upwind}, \eqref{eq_R_FC}. Let \mbox{$f \in (P_3)^n$}. On simplicial TI-meshes, the truncation error on $f$ is zero if $\bm{A} \cdot \nabla f \equiv 0$. For $d \le 3$, this was proved in \cite{Nishikawa2017,Bakhvalov2017ufc}, and for $d \in \mathbb{N}$ this is proved below (see Corollary~\ref{th:FC_is_3exact}).

Let \mbox{$f \in (P_3)^n$} be such that $\bm{A} \cdot \nabla f \equiv 0$. For each $m \in \mathbb{N}$ we have
\begin{equation*}
\begin{gathered}
m^d \bar{\epsilon}(f,\Pi) = \sum\limits_{j \in \mathcal{M}^m} |K_j|\ \epsilon_j(f,\Pi) =
\sum\limits_{\langle j, k \rangle\ :\ j \in \mathcal{M}^m, k \not \in \mathcal{M}^m, k \in N(j)} F_{jk}[\Pi f].
\end{gathered}
\end{equation*}
Using the same arguments as above, we come to the conclusion that on an unstructured mesh the mean truncation error on $f$ is zero.


Now consider the flux correction method with the {\it ``divergence formulation''} of the unsteady term. Let $f \in (P_3)^n$. Denote $g = (\bm{A} \cdot \nabla) f \in (P_2)^n$. By construction,
$$
|K_j|\ \epsilon_j(f, \Pi) = - |K_j|\sum\limits_{k \in N(j)} m_{jk} g(\bm{r}_k) + \sum\limits_{k \in N(j)} F_{jk}[\Pi f] =
$$
$$
= -\sum\limits_{k \in N(j)} \phi_{jk}[\bm{v}_j[\Pi g]]+ \sum\limits_{k \in N(j)} F_{jk}[\Pi f] = 
\sum\limits_{k \in N(j)} (F_{jk}[\Pi f] - \phi_{jk}[\bm{w}[g]])
$$
where $\bm{v}_j$ and $\bm{w}[g]$ are defined by \eqref{eq_def_vj} and \eqref{eq_def_wl1}, correspondingly. The last equality in this chain is by  \eqref{eq_def_wl0}. Since $\phi_{jk}\equiv -\phi_{kj}$, we have
$$
m^d \bar{\epsilon}(f, \Pi) = \sum\limits_{\langle j, k \rangle\ :\ j \in \mathcal{M}^m, k \not \in \mathcal{M}^m, k \in N(j)} (F_{jk}[\Pi f] - \phi_{jk}[\bm{w}[g]])
$$
It remains to use the auxiliary mesh argument.

Finally, consider the flux correction method with the use of the  {\it ``extended Galerkin formula''} \eqref{eqFC}. 
We need the following quadrature rule on a $d$-dimensional simplex $e$:
\begin{equation}
\frac{1}{|e|} \int\limits_e f(\bm{r}) dv \approx \alpha \sum\limits_{j=0}^d f(\bm{r}_j) + \beta \sum\limits_{0 \le j < k \le d} f\left(\frac{\bm{r}_j+\bm{r}_k}{2}\right).
\label{eq_numint_1}
\end{equation}
\begin{equation}
\alpha = \frac{2-d}{(d+1)(d+2)}, \quad \beta = \frac{4}{(d+1)(d+2)}.
\label{eq_numint_2}
\end{equation}
Here $\bm{r}_j$ are the radius-vectors of the vertices of $e$ and $|e|$ is the volume of $e$.

\begin{lemma}
The quadrature rule \eqref{eq_numint_1}--\eqref{eq_numint_2} is 2-exact.
\end{lemma}
\begin{proof}
Since a quadratic polynomial of $d$ variables is well-defined by its values at vertices and edge centers of a $d$-simplex, there is a 2-exact quadrature rule of the form \eqref{eq_numint_1} with some $\alpha$ and $\beta$. For $d=1$, the values \eqref{eq_numint_2} yield the Simpson rule. For $d \ge 2$, the values \eqref{eq_numint_2} can be easily verified by considering the functions $f=1$ and $f=xy$ on the simplex with the vertices $(0,\ldots, 0)$, $(1, \ldots, 0)$, $(0, 1, \ldots, 0)$, ..., $(0, \ldots, 0, 1)$.
\end{proof}

\begin{lemma}\label{th:lemma4}
Let $e$ be a simplex and $p \in P_2$. Then
$$
\frac{1}{d(d+1)} \sum\limits_{j,k \in e, j \ne k} s_{jk}[\Pi p] = \frac{1}{|e|} \int\limits_e p(\bm{r}) dv.
$$
\end{lemma}
\begin{proof}
By construction,
$$
s_{jk}[\Pi p] = p(\bm{r}_j) - \frac{d}{d+2} \left(\frac{\bm{r}_k - \bm{r}_j}{2} \cdot \nabla\right)^2 p.
$$
Evaluating the second derivative of a quadratic polynomial by the 3-point formula we obtain
$$
s_{jk}[\Pi p] + s_{kj}[\Pi p] = p(\bm{r}_j) + p(\bm{r}_k) - \frac{2d}{d+2} \left(p(\bm{r}_j) + p(\bm{r}_k) - 2p\left(\frac{\bm{r}_j + \bm{r}_k}{2}\right)\right).
$$
Taking the sum over all ordered pairs of the simplex vertices we get
$$
\sum\limits_{j,k \in e; j \ne k} s_{jk}[\Pi p] = d\frac{2-d}{d+2}\sum\limits_{j \in e} p(\bm{r}_j) + \frac{4d}{d+2} \sum\limits_{0 \le j < k \le d} p\left(\frac{\bm{r}_j+\bm{r}_k}{2}\right).
$$
The right-hand side of this identity coincides with the quadrature rule \eqref{eq_numint_1}--\eqref{eq_numint_2} multiplied by $d(d+1)$.
\end{proof}

Now we are ready to prove that the zero mean error on $f \in (P_3)^n$ extends from simplicial TI-meshes to unstructured meshes.

Let $f \in (P_3)^n$, $g = \bm{A} \cdot \nabla f$. Then
\begin{equation}
\begin{gathered}
|K_j|\, \epsilon_j(f, \Pi) = -\sum\limits_{k \in N(j)}  v_{jk} \, s_{jk}[g] 
+ \sum\limits_{k \in N(j)} F_{jk}[f].
\end{gathered}
\label{eq_aux_fc_new}
\end{equation}
Here $v_{jk}$ and $s_{jk}$ are given by \eqref{eq_def_vjk} and \eqref{eq_def_sjk}, correspondingly. It is easy to see that 
\begin{equation}
v_{jk} = \frac{1}{d(d+1)} \sum\limits_{e \ni j, k} v_e
\label{eq_alt_vjk}
\end{equation}
where the sum is over simplexes $e$ containing nodes $j$ and $k$, and $v_e$ is the volume of $e$.
Now we take the sum of \eqref{eq_aux_fc_new} over $j \in \mathcal{N}_m$. Representing the first term on the right-hand side as the sum over simplexes we obtain
\begin{equation}
\begin{gathered}
\sum\limits_{j \in \mathcal{N}^m} |K_j|\, \epsilon_j(f, \Pi) = -\frac{1}{d(d+1)}\sum\limits_{e \in \mathcal{E}}  v_e \sum\limits_{j,k \in e; j \ne k; j \in \mathcal{N}^m}   s_{jk}[g] 
+ \sum\limits_{j \in \mathcal{N}^m} \sum\limits_{k \in N(j)} F_{jk}[f].
\end{gathered}
\label{eq_aux_fc_new1}
\end{equation}
By Lemma~\ref{th:lemma4}, the first term on the right-hand side of \eqref{eq_aux_fc_new1} is the integral of $g$ over $V^m$, up to several layers at its boundary. In the second term, the numerical fluxes along edges, both vertices of which belong to $\mathcal{N}^m$, negate each other. Therefore, the expression on the right-hand side of \eqref{eq_aux_fc_new1} depends only on the mesh near the boundary of the cube with edge $m$. It remains to use the auxiliary mesh argument.

\section{Mean truncation error. Rigorous approach}
\label{sect:check2}

In the previous section, we considered several finite-volume methods. We showed that the $(p+1)$-exactness on TI-meshes yields the zero mean error property on unstructured meshes. However, we have not proved the $(p+1)$-exactness itself. Also we used some assumptions without a rigorous proof, for instance, that a mesh can be linked with a uniform mesh while preserving the properties of the scheme.

In this section, we consider the same finite-volume schemes. Using their form explicitly, we prove the zero mean error property without additional assumptions. As a corollary, we establish the 3-exactness of the flux correction method on simplicial TI-meshes for $d \in \mathbb{N}$, which was previously known only for $d \le 3$ (see \cite{Nishikawa2017,Bakhvalov2017ufc}).

\subsection{Schemes with the polynomial reconstruction}
\label{sect:polynomial}

Here we consider the cell-centered finite-volume schemes with the polynomial reconstruction on simplicial meshes. The vertex-centered case may be treated similarly. We have $\mathcal{M} = \mathcal{E}$, and $\Pi$ is the operator taking a function $f \in C(\mathbb{R}^d)$ to its average over mesh elements.

By construction, a finite-volume scheme with the $p$-th order polynomial reconstruction is $p$-exact.

\begin{proposition}\label{th:av:fv}
A finite-volume scheme with the $p$-th order polynomial reconstruction has zero mean error on the $(p+1)$-th order polynomials.
\end{proposition}
\begin{proof}
Let $f \in (P_{p+1})^n$. Since $(\bm{A} \cdot \nabla) f \in (P_p)^n$, we have
$$
|K_j|\ \epsilon_j(f, \Pi) = -\int\limits_{K_j} (\bm{A}\cdot \nabla) f dv + \sum\limits_{k \in N(j)} F_{jk}[\Pi f].
$$
By the Gauss theorem
$$
|K_j|\ \epsilon_j(f, \Pi) = \sum\limits_{k \in N(j)} \epsilon_{jk}(f, \Pi)
$$
with
$$
\epsilon_{jk}(f, \Pi) = F_{jk}[\Pi f] - (\bm{A} \cdot \bm{n}_{jk}) \int\limits_{\D K_j \cap \D K_k} f(\bm{r}) ds.
$$

For $j \in \mathcal{E}$ and $k \in N(j)$, denote by $\langle j, k \rangle$ the face sharing cells $j, k \in \mathcal{E}$. We say that face $\langle k, j \rangle$ is inverse to $\langle j, k \rangle$. We also say that a face $\langle j, k \rangle$ is equivalent to $\langle \tilde{j}, \tilde{k} \rangle$ if there exists an integer vector $\bm{r}_T$ such that $\bm{r}_{\tilde{j}} = \bm{r}_j + \bm{r}_T$ and $\bm{r}_{\tilde{k}} = \bm{r}_k + \bm{r}_T$. 

Write
\begin{equation*}
\sum\limits_{j \in \mathcal{E}^1} |K_j|\ \epsilon_j(f, \Pi) = \sum\limits_{j \in \mathcal{E}^1} \sum\limits_{k \in N(j)} \epsilon_{jk}(f, \Pi).
\end{equation*}
The double sum on the right-hand side may be represented as a sum of several terms of the form 
$$
\epsilon_{jk}(f, \Pi) + \epsilon_{\tilde{k} \tilde{j}}(f, \Pi)
$$
where $\langle \tilde{j}, \tilde{k} \rangle$ is equivalent to $\langle j, k \rangle$. We claim that each expression of this form is equal to zero.

By the translational invariance of the scheme,
$$
F_{\tilde{j}\tilde{k}}[\Pi f] = F_{jk}[\Pi \tilde{f}], \quad \tilde{f}(\bm{r}) = f(\bm{r} - \bm{r}_T).
$$
Since $\bm{n}_{jk} = \bm{n}_{\tilde{j}\tilde{k}}$, $\bm{n}_{jk} = -\bm{n}_{kj}$ and $F_{jk} = -F_{kj}$, we have
$$
\epsilon_{jk}(f, \Pi) + \epsilon_{\tilde{k} \tilde{j}}(f, \Pi) 
= 
F_{jk}[\Pi (f- \tilde{f})] - (\bm{A} \cdot \bm{n}_{jk}) \int\limits_{\D K_j \cap \D K_k} (f(\bm{r}) - \tilde{f}(\bm{r})) ds.
$$
Since $f \in (P_{p+1})^n$, there holds $f - \tilde{f} \in (P_p)^n$. By the $p$-exactness of the numerical flux, the expression on the right-hand size is zero. This concludes the proof.
\end{proof}

\subsection{Multislope cell-centered method}
\label{sect:rig:BBR3}

Here $\Pi$ is the operator taking a function $f \in C(\mathbb{R}^d)$ to its point values at element centers.

\begin{proposition}
Let the points $\bm{r}_{jk}^{\pm}$ are defined for each $j \in \mathcal{N}$ and $k \in N(j)$. Then the scheme BBR3 is 1-exact and has zero mean error on the second order polynomials.
\end{proposition}
\begin{proof}
The first claim is by construction, and the proof of the second claim repeats the one for schemes with a polynomial reconstruction.
\end{proof}

We are concerned with the behavior of this scheme on a translationally-invariant triangular mesh (for instance, on a regular triangular mesh). Let $\bm{e}_1$, $\bm{e}_2$, and $\bm{e}_2 - \bm{e}_1$ be the vectors of the mesh edges. All triangles have the same area, denote it by $V/2$.

Consider two adjacent cells $j, k \in \mathcal{E}$. Let $\bm{e} = 3(\bm{r}_k - \bm{r}_j)$. Denote by $\hat{j}$ the element that is the translation of element $k$ by $-\bm{e}$. Denote by $\hat{k}$ the element that is the translation of $j$ by $\bm{e}$. Then \eqref{eq_R_BBR3} reduces to
\begin{equation}
\begin{gathered}
\mathcal{R}_{jk}[u] = -\frac{1}{12} u_{\hat{j}} + \frac{3}{4} u_j + \frac{1}{3} u_k,
\\
\mathcal{R}_{kj}[u] = -\frac{1}{12} u_{\hat{k}} + \frac{3}{4} u_k + \frac{1}{3} u_j.
\end{gathered}
\label{eq_ulr_1}
\end{equation}

\begin{proposition}\label{th:BBR:2exact}
On a triangular TI-mesh, the scheme BBR3 is 2-exact.
\end{proposition}
\begin{proof}
Let $j \in \mathcal{E}$. Since BBR3 is 1-exact by construction, it is enough to consider functions $f \in (P_2)^n$ such that $f(\bm{r}_j) = 0$ and $\nabla f(\bm{r}_j) = 0$. Let $u = \Pi f$. Then the restriction of $f$ to the line passing through $\bm{r}_j$ and $\bm{r}_k$ has the form $c |\bm{r}-\bm{r}_L|^2$ for some $c \in \mathbb{R}^n$. Thus
$$
u_{\hat{j}} = 4u_k, \quad u_{\hat{k}} = 9u_k.
$$
Taking this to \eqref{eq_ulr_1} we get $\mathcal{R}_{jk}[u] = \mathcal{R}_{kj}[u] = 0$. Thus $F_{jk}[u] = 0$, and the 2-exactness easily follows.
\end{proof}

Note that on a translationally-invariant tetrahedral mesh, the scheme BBR3 is not 2-exact.

\begin{proposition}\label{th:av:BBR}
Let $f \in (P_3)^n$ such that $\bm{A} \cdot \nabla f = 0$. On a triangular TI-mesh, the mean truncation error of the scheme BBR3 on $f$ equals zero.
\end{proposition}

\begin{proof}
For each $j \in \mathcal{E}$ and $k \in N(j)$, denote by $\bm{r}_{jk}$ the radius-vector of the middle of the edge shared by $K_j$ and $K_k$. 

Each pair of adjacent triangles forms a periodic cell. Let cells $a \in \mathcal{E}$ and $b \in \mathcal{E}$ share an edge with the vector $\bm{e}_2 - \bm{e}_1$, and $\bm{r}_0$ be the middle of this edge, see Fig.~\ref{fig:bbr_regular}. 

\begin{figure}[t]
\centering
\includegraphics[width=0.5\linewidth]{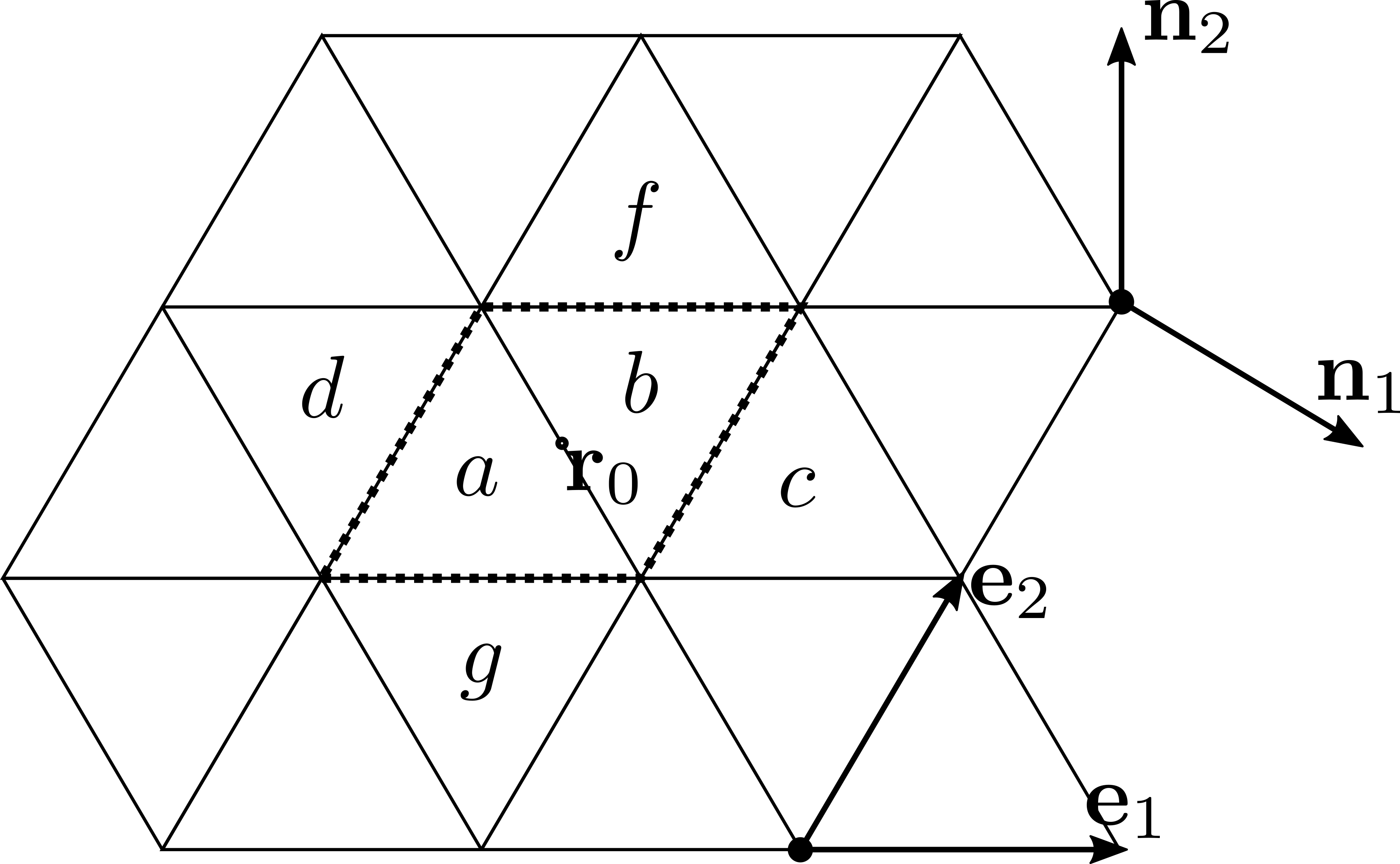}
\caption{A regular triangular mesh with the notation used in Proposition~\ref{th:av:BBR}}\label{fig:bbr_regular}
\end{figure}

Using $F_{ab} \equiv -F_{ba}$ and $\bm{A} \cdot \nabla f = 0$ we get
$$
|K_a|\ \epsilon_a(f, \Pi) + |K_b|\ \epsilon_b(f, \Pi) = (F_{ad}[\Pi f] + F_{bc}[\Pi f]) + (F_{ag}[\Pi f] + F_{bf}[\Pi f]).
$$
We need to prove that this equals zero.

First consider the edge shared by cells $b$ and $c$, which is parallel to $\bm{e}_2$. The translation above is $\bm{e} = 2\B{e}_1 - \bm{e}_2$. Then \eqref{eq_ulr_1} yields
\begin{equation}
\begin{gathered}
\mathcal{R}_{bc}[\Pi f] = f(\bm{r}_{bc}) - \frac{1}{72} \left((2\B{e}_1 - \bm{e}_2) \cdot \nabla\right)^2 f(\bm{r}_{bc})
+ \frac{5}{648}\left((2\B{e}_1 - \bm{e}_2) \cdot \nabla\right)^3 f(\bm{r}_{bc}),
\\
\mathcal{R}_{cb}[\Pi f] = f(\bm{r}_{bc}) - \frac{1}{72} \left((2\B{e}_1 - \bm{e}_2) \cdot \nabla\right)^2 f(\bm{r}_{bc})
- \frac{5}{648}\left((2\B{e}_1 - \bm{e}_2) \cdot \nabla\right)^3 f(\bm{r}_{bc}).
\label{eq_aux_1}
\end{gathered}
\end{equation}
Taking this to \eqref{upwind} we get
\begin{equation*}
\begin{gathered}
F_{bc}[\Pi f] = \bm{A}\cdot\B{n}_{1}\left(f(\bm{r}_{bc}) - \frac{1}{72} \left((2\B{e}_1 - \bm{e}_2) \cdot \nabla\right)^2 f(\bm{r}_{bc})\right) 
+
\\
+ \frac{5}{648} |\bm{A} \cdot \bm{n}_{1}| \left((2\B{e}_1 - \bm{e}_2) \cdot \nabla\right)^3 f(\bm{r}_{bc}).
\end{gathered}
\end{equation*}
Denote $d_1 = F_{bc}[\Pi f] - F_{da}[\Pi f]$, then
\begin{equation*}
\begin{gathered}
d_1 = \bm{A} \cdot \bm{n}_1 \left(f(\bm{r}_{bc}) - \frac{1}{72} \left((2\B{e}_1 - \bm{e}_2) \cdot \nabla\right)^2 f(\bm{r}_{bc}) 
 - f(\bm{r}_{ad}) + \frac{1}{72} \left((2\B{e}_1 - \bm{e}_2) \cdot \nabla\right)^2 f(\bm{r}_{ad})\right).
\end{gathered}
\end{equation*}
Using $f \in (P_3)^n$, $\bm{r}_{bc} = \bm{r}_0 + \bm{e}_1/2$, $\bm{r}_{ad} = \bm{r}_0 - \bm{e}_1/2$, one can verify that
$$
d_1 = (\bm{A} \cdot \bm{n}_1) (\bm{e}_1 \cdot \nabla) \left[
 f(\bm{r}_{0}) - \frac{1}{72}
\left[ (\bm{e}_1 \cdot \nabla)^2 - 4(\bm{e}_1 \cdot \nabla) (\bm{e}_2 \cdot \nabla)  +
(\bm{e}_2 \cdot \nabla)^2\right] f(\bm{r}_{0})\right].
$$
Similarly, for $d_2 = F_{bf}[\Pi f] - F_{ga}[\Pi f]$ there holds
$$
d_2 = (\bm{A} \cdot \bm{n}_2) (\bm{e}_2 \cdot \nabla) \left[
 f(\bm{r}_{0}) - \frac{1}{72}
\left[ (\bm{e}_1 \cdot \nabla)^2 - 4(\bm{e}_1 \cdot \nabla) (\bm{e}_2 \cdot \nabla)  +
(\bm{e}_2 \cdot \nabla)^2\right] f(\bm{r}_{0})\right].
$$
Since the vectors $\{\bm{n}_1/V, \bm{n}_2/V\}$ form the biorthogonal basis to $\{\bm{e}_1, \bm{e}_2\}$, then we have
$$
(\bm{A} \cdot \bm{n}_1) (\bm{e}_1 \cdot \nabla) + (\bm{A} \cdot \bm{n}_2) (\bm{e}_2 \cdot \nabla) = V(\bm{A} \cdot \nabla)
$$
and thus $d_1 + d_2 = 0$. This concludes the proof. 
\end{proof}

If $\bm{A} \cdot \nabla f \ne 0$, it can be shown that
\begin{equation*}
\begin{gathered}
|K_a|\  \epsilon_a(f, \Pi) + |K_b|\  \epsilon_b(f, \Pi) =  
\\
=
-
\frac{V}{72}
\left[(\bm{e}_1 \cdot \nabla)^2 + (\bm{e}_2 \cdot \nabla)^2 +
((\bm{e}_2 - \bm{e}_1) \cdot \nabla)^2 \right] (\bm{A} \cdot \nabla)  u(\bm{r}_{0}).
\end{gathered}
\end{equation*}

\subsection{Edge-based schemes}

Here $\Pi$ is the operator taking a function $f \in C(\mathbb{R}^d)$ to its values at mesh nodes.

\begin{proposition}
An edge-based scheme has zero mean error on the second order polynomials.
\label{th:prop:EB}
\end{proposition}

The proof resembles the one of Proposition~\ref{th:av:fv}, but instead of the exact flux we use the numerical flux by the flux correction method.

\begin{proof}
Let $f \in (P_{2})^n$. Write
$$
\epsilon_j(f, \Pi) = - (\bm{A}\cdot \nabla) f(\bm{r}_j) + \frac{1}{|K_j|} \sum\limits_{k \in N(j)} F_{jk}[\Pi f].
$$
From \eqref{eq_1_exact},
$$
\nabla f(\bm{r}_j) = \frac{1}{|K_j|} \sum\limits_{k \in N(j)} \left(f(\bm{r}_j) + \frac{\bm{r}_k - \bm{r}_j}{2} \cdot \nabla f(\bm{r}_j)\right) \bm{n}_{jk}.
$$
Hence,
$$
\epsilon_j(f, \Pi) = \frac{1}{|K_j|} \sum\limits_{k \in N(j)} \epsilon_{jk}(f, \Pi),
$$
$$
\epsilon_{jk}(f, \Pi) = F_{jk}[\Pi f] - (\bm{A} \cdot \bm{n}_{jk})\left(f(\bm{r}_j) + \frac{\bm{r}_k - \bm{r}_j}{2} \cdot \nabla f(\bm{r}_j)\right).
$$
For a quadratic polynomial $f$ it is easy to check that
$$
f(\bm{r}_j) + \frac{\bm{r}_k - \bm{r}_j}{2} \cdot \nabla f(\bm{r}_j) = f(\bm{r}_k) - \frac{\bm{r}_k - \bm{r}_j}{2} \cdot \nabla f(\bm{r}_k),
$$
therefore, $\epsilon_{jk}(f, \Pi) = - \epsilon_{kj}(f, \Pi)$. The rest of the proof repeats the one of Proposition~\ref{th:av:fv}.
\end{proof}

\subsection{Flux correction method}
\label{sect:FC_strict}

In this section we prove the zero mean error condition for the ``extended Galerkin formula'' approach \eqref{eqFC}. For the steady flux correction method, the zero mean truncation error on a polynomial $f \in (P_3)^n$ such that $\bm{A} \cdot \nabla f = 0$, is a direct corollary.

Let $v$ be a $d$-dimensional simplex, $|v|$ be its volume, $\bm{r}_v$ be its mass center, and $M_{v,\bm{e}}$ be defined as
$$
\bm{r}_v = \frac{1}{|v|} \int\limits_v \bm{r} dV, \quad
M_{v,\bm{e}} = \int\limits_v (\bm{e} \cdot (\bm{r}-\bm{r}_v))^2 dV.
$$

\begin{lemma}\label{th:lemma12}
Let $\bm{r}_j$, $j = 0, \ldots, d$ be the vertices of $v$. Then
\begin{equation}
M_{v,\bm{e}}  =  \frac{1}{(d+1)(d+2)} |v|\sum\limits_{j=0}^d (\bm{e} \cdot (\bm{r}_j - \bm{r}_v))^2,
\label{eq_lemma12_a}
\end{equation}
\begin{equation}
M_{v,\bm{e}}  =  \frac{1}{(d+1)^2(d+2)} |v|\sum\limits_{0 \le j < k \le d}^d (\bm{e} \cdot (\bm{r}_j - \bm{r}_k))^2.
\label{eq_lemma12_b}
\end{equation}
\end{lemma}
\begin{proof}
Denote $x_j = \bm{e}\cdot (\bm{r}_j - \bm{r}_v)$. 
Using the quadrature rule \eqref{eq_numint_1}--\eqref{eq_numint_2} we get
$$
\frac{M_{v,\bm{e}}}{|v|} = \beta \sum\limits_{0 \le j < k \le d}  \left(\frac{x_j + x_k}{2}\right)^2 + \alpha \sum\limits_{j=0}^d x_j^2.
$$
Note that
$$
0 = \left(\sum\limits_{j=0}^d x_j\right)^2 = \sum\limits_{j=0}^d x_j^2 + 2\sum\limits_{0 \le j < k \le d} x_j x_k = \sum\limits_{0 \le j < k \le d} (x_j+x_k)^2 - (d-1)\sum\limits_{j=0}^d x_j^2.
$$
Then
$$
\frac{M_{v,\bm{e}}}{|v|} = \frac{\beta}{4} \cdot (d-1) \sum\limits_{j=0}^d x_j^2 + \alpha \sum\limits_{j=0}^d x_j^2 = \left(\alpha + \frac{d-1}{4}\beta\right) \sum\limits_{j=0}^d x_j^2,
$$
which coincides with \eqref{eq_lemma12_a}.

To prove \eqref{eq_lemma12_b}, note that
$$
\sum\limits_{0 \le j < k \le d}^d (\bm{e} \cdot (\bm{r}_j - \bm{r}_k))^2
=
\sum\limits_{0 \le j < k \le d}^d (x_j - x_k)^2
=
$$
$$
= - \sum\limits_{0 \le j < k \le d}^d (x_j + x_k)^2 + 2 \sum\limits_{0 \le j < k \le d}^d (x_j^2 + x_k^2) = \left[-(d+1) + 2d\right] \sum\limits_{j=0}^d x_j^2.
$$
It remains to use \eqref{eq_lemma12_a}.
\end{proof}

\begin{lemma}\label{th:lemma15}
For each $x_1, \ldots, x_d \in \mathbb{R}$ there holds
\begin{equation}
\begin{gathered}
\sum\limits_{l=1}^d (x_l - x_0)^3 \equiv A_1 x_{\gamma} \left(x_0^2 + \sum\limits_{l=1}^d x_l^2\right)  + A_2 \sum\limits_{l=1}^d (x_l-x_{\gamma})^3 + 
\\
+ A_3 \left[\frac{3-d}{d(d+1)} \sum\limits_{l=1}^d x_l^3 + \frac{4}{d(d+1)} \sum\limits_{1 \le l < m \le d} \left(\frac{x_l+x_m}{2}\right)^3\right],
\end{gathered}
\end{equation}
where $x_0 = -\sum_l x_l$, $x_{\gamma} = d^{-1} \sum_l x_l = -d^{-1} x_0$, 
\begin{equation}
A_1 = \frac{3}{2}d(d+1), \quad A_2 = - \frac{1}{4} d^2(d-3), \quad A_3 = \frac{1}{2} d(d+1)^2(2-d). 
\label{eq_expr_A123}
\end{equation}
\end{lemma}
\begin{proof}
The expressions on both sides of the identity are symmetric third order polynomials with respect to $x_1, \ldots, x_d$. Equating the coefficients at $x_1^3$, $x_1^2 x_2$ (if $d \ge 2$), and $x_1x_2x_3$ (if $d \ge 3$), we get 
\begin{equation*}
\begin{gathered}
7+d = 2d^{-1} A_1 + d^{-2}(d-1)(d-2) A_2 + \frac{5-d}{2d(d+1)} A_3,
\\
3(d+4) = 4d^{-1} A_1 + 3d^{-2}(2-d) A_2 + \frac{3}{2d(d+1)}A_3,
\\
6(d+3) = 6d^{-1}A_1 + 12d^{-2} A_2.
\end{gathered}
\end{equation*}
These equations are satisfied by the above values of $A_1$, $A_2$, $A_3$.
\end{proof}

\begin{lemma}\label{th:lemma14}
For each $\bm{e} \in \mathbb{R}^d$ there holds
$$
\bm{S}(\bm{e}) :=
\sum\limits_{j \in \mathcal{N}^1} \sum\limits_{k \in N(j)} (\bm{e}\cdot(\bm{r}_k - \bm{r}_j))^3 \bm{n}_{j, k} = -12(d+1) \bm{e} \sum\limits_{v \in \mathcal{E}^1} M_{v,\bm{e}}.
$$
\end{lemma}
\begin{proof}
Let $\Gamma(v)$ be the set of faces of simplex $v \in \mathcal{E}$, and $\gamma(v,j) \in \Gamma(v)$ be the face of $v$ that does not contain node $j$. For $\gamma \in \Gamma(v)$ put
$$
\bm{n}_{\gamma} = \int\limits_{\gamma} \bm{n} ds,
$$
where $\bm{n}$ is the unit normal pointing outside $v$. By $\bm{r}_{\gamma}$ denote the radius-vector of the mass center of $\gamma$.

By Lemma~\ref{th:B2} we get
\begin{equation}
\bm{n}_{jk} =
\frac{1}{d(d+1)} \sum\limits_{e \ni j, k}\ \biggl(\int\limits_{\gamma(e,j)} \bm{n} ds - \int\limits_{\gamma(e,k)} \bm{n} ds\biggl) =
 \frac{2}{d(d+1)} \sum\limits_{e \ni j, k}\ \int\limits_{\gamma(e,j)} \bm{n} ds.
\label{eq_def_v_n}
\end{equation}
Then
$$
\bm{S}(\bm{e}) = \frac{2}{d(d+1)}  \sum\limits_{j \in \mathcal{N}^1} \sum\limits_{k \in N(j)} (\bm{e}\cdot(\bm{r}_k - \bm{r}_j))^3 \sum\limits_{v \ni j, k} \bm{n}_{\gamma(v,j)}.
$$
Changing the order of summation we get
$$
\bm{S}(\bm{e}) = \frac{2}{d(d+1)} \sum\limits_{v \in \mathcal{E}^1} \sum\limits_{\gamma \in \Gamma(v)} \bm{n}_{\gamma} \sum\limits_{l=1}^d (\bm{e}\cdot(\bm{r}_{\gamma,l} - \bm{r}_{v,\gamma,0}))^3.
$$
Here $\bm{r}_{\gamma,l}$ is the radius-vector of $l$-th vertex of $\gamma$, where $l=1,\ldots, d$, and $\bm{r}_{v,\gamma,0}$ is the radius-vector of the vertex of $v$ that does not belong to $\gamma$. Now use Lemma~\ref{th:lemma15} with \mbox{$x_l = \bm{e}\cdot(\bm{r}_{\gamma,l} - \bm{r}_v)$}, $x_0 = \bm{e}\cdot(\bm{r}_{v,\gamma,0} - \bm{r}_v)$. We obtain

\begin{equation*}
\begin{gathered}
\frac{d(d+1)}{2}\bm{S}(\bm{e}) = 
A_1 \sum\limits_{v \in \mathcal{E}^1} \sum\limits_{\gamma \in \Gamma(v)} \bm{n}_{\gamma} (\bm{e} \cdot (\bm{r}_{\gamma} - \bm{r}_v)) \sum\limits_{l=0}^d (\bm{e}\cdot(\bm{r}_l - \bm{r}_v))^2
+
\\
+
A_2 \sum\limits_{v \in \mathcal{E}^1} \sum\limits_{\gamma \in \Gamma(v)} \bm{n}_{\gamma}  \sum\limits_{l=1}^d (\bm{e}\cdot(\bm{r}_{v,\gamma,l} - \bm{r}_{\gamma}))^3 
+ A_3 \sum\limits_{v \in \mathcal{E}^1} \sum\limits_{\gamma \in \Gamma(v)} \bm{n}_{\gamma} \Psi_\gamma[(\bm{e}\cdot(\bm{r} - \bm{r}_v))^3],
\\
\Psi_\gamma[f] =  \frac{3-d}{d(d+1)} \sum\limits_{l=1}^d f(\bm{r}_{\gamma,l}) + \frac{4}{d(d+1)} \sum\limits_{l=1}^{d-1} \sum\limits_{m=l+1}^d f\left(\frac{\bm{r}_{\gamma,l}+\bm{r}_{\gamma,m}}{2}\right).
\end{gathered}
\end{equation*}

The sum in $l$ in the term with $A_1$ does not depend on $\gamma$. By Lemma~\ref{th:lemma12} it equals $(d+1)(d+2)|v|^{-1} M_{v,\bm{e}}$. The rest of the sum in $\gamma$ is the gradient of \mbox{$\bm{e} \cdot (\bm{r}-\bm{r}_v)$} computed by the Gauss formula and multiplied by $|v|$. Thus, the term at $A_1$ equals $(d+1)(d+2) \bm{e} \sum_v M_{v,\bm{e}}$.

The sum in $l$ in the term with $A_2$ depends on the face but not on its orientation. On each pair of equivalent faces, it also coincides. Thus, after the terms related to both sides of one face negate each other.

The functional $\Psi_{\gamma}[f]$ is a 2-exact quadrature rule to approximate the integral average of $f$ on $\gamma$. Write
$$
\Psi_\gamma[(\bm{e}\cdot(\bm{r} - \bm{r}_v))^3] =
\Psi_{\gamma}[(\bm{e}\cdot(\bm{r} - \bm{r}_{\gamma}))^3] + 
\Psi_\gamma[(\bm{e}\cdot(\bm{r} - \bm{r}_v))^3 - (\bm{e}\cdot(\bm{r} - \bm{r}_{\gamma}))^3].
$$
The first term on the right-hand side drops after the sum by two sides of $\gamma$, and the second one equals to the integral average of the argument on $\gamma$. Thus 
$$
\sum\limits_{v \in \mathcal{E}^1} \sum\limits_{\gamma \in \Gamma(v)} \bm{n}_{\gamma} \Psi_\gamma[(\bm{e}\cdot(\bm{r} - \bm{r}_v))^3] = 
\sum\limits_{v \in \mathcal{E}^1} \sum\limits_{\gamma \in \Gamma(v)}  \int\limits_\gamma (\bm{e}\cdot(\bm{r} - \bm{r}_v))^3 \bm{n} ds
= \sum\limits_{v \in \mathcal{E}^1} 3\B{e} M_{v,\bm{e}}.
$$
The last identity is by the Gauss formula.


Finally, we obtain
$$
\bm{S}(\bm{e}) = \frac{2}{d(d+1)}\left[(d+1)(d+2)A_1 + 3A_3\right] \bm{e}\sum\limits_{v \in \mathcal{E}^1} M_{v,\bm{e}}.
$$
The use of the values \eqref{eq_expr_A123} for $A_1$ and $A_3$ gives the identity to prove.
\end{proof}

\begin{proposition}
The flux correction method with the ``extended Galerkin formula'' has zero mean error on the third order polynomials.
\end{proposition}
\begin{proof}
Let $i_{\alpha} \in \mathbb{R}^n$, $\alpha = 1, \ldots, n$, be the vectors of the standard basis in $\mathbb{R}^n$. It is enough to consider the average truncation error on $f = i_{\alpha} (\bm{r}\cdot \bm{e})^3$, $\bm{e} \in \mathbb{R}^d$. For brevity, denote $x = \bm{r}\cdot \bm{e}$, $x_j = \bm{r}_j\cdot \bm{e}$. Therefore, we need to prove that
\begin{equation}
\sum\limits_{j \in \mathcal{N}^1} |K_j|\, \epsilon_j(i_{\alpha} x^3, \Pi) = 0.
\label{eq_pp1_fc_a}
\end{equation}

By 2-exactness,
$$
\epsilon_j(i_{\alpha} x^3, \Pi) = \epsilon_j(i_{\alpha} (x-x_j)^3, \Pi).
$$
Expanding the definition for the truncation error and taking the sum in $j$ we get
$$
\sum\limits_{j \in \mathcal{N}^1} |K_j|\, \epsilon_j(i_{\alpha} x^3, \Pi) = S_1 + S_2
$$
with
$$
S_1 =  \sum\limits_{j \in \mathcal{N}^1}\sum\limits_{k \in N(j)}  v_{jk} \, s_{jk}[3\Pi((-\bm{A} \cdot \bm{e}\,i_{\alpha})(x-x_j)^2)],
$$
$$
S_2 = \sum\limits_{j \in \mathcal{N}^1} \sum\limits_{k \in N(j)} F_{jk}[i_{\alpha}\Pi((x-x_j)^3)].
$$

Consider $S_1$ first. By construction,
$$
s_{jk}[i_{\alpha}\Pi((x-x_j)^2)] = - \frac{d}{2(d+2)} (x_k - x_j)^2 i_{\alpha}.
$$
Using the expression \eqref{eq_alt_vjk} for $v_{jk}$ we get
$$
S_1 = - \frac{d}{2(d+2)} \frac{3}{d(d+1)}(\bm{A} \cdot \bm{e}\,i_{\alpha}) \sum\limits_{\langle j,k \rangle} \sum\limits_{v \ni j, k} |v| (x_k - x_j)^2 =
$$
$$
= -\frac{3}{2(d+1)(d+2)}(\bm{A} \cdot \bm{e}\,i_{\alpha}) \sum\limits_{v \in \mathcal{E}^1} |v| \sum\limits_{j,k \in v, j \ne k} (x_k - x_j)^2.
$$
In the last sum, each edge is counted twice, namely, as $\langle j,k \rangle$ and as $\langle k,j \rangle$. By Lemma~\ref{th:lemma12} we get
$$
S_1 = - 3(d+1) (\bm{A} \cdot \bm{e}\,i_{\alpha}) \sum\limits_{v \in \mathcal{E}^1} M_{v,\bm{e}}.
$$

Now consider $S_2$. Represent $x-x_j = (x-x_{jk}) + (x_{jk}-x_j)$, then
$$
(x-x_j)^3 = (x-x_{jk})^3 + 3(x-x_{jk})^2 (x_{jk}-x_j) + 3(x-x_{jk}) (x_{jk}-x_j)^2  + (x_{jk}-x_j)^3.
$$
By definition, 
$$
F_{jk}[i_{\alpha}\Pi((x-x_{jk})^2)] = (\bm{A} \cdot \bm{n}_{jk}\,i_{\alpha} ) \left[(x_j-x_{jk})^2 + 2(x_j-x_{jk}) (x_{jk}-x_{j})\right] =
$$
$$
= -\frac{(x_j-x_k)^2}{4}(\bm{A}\cdot \bm{n}_{jk}\,i_{\alpha} ).
$$
By 1-exactness
$$
F_{jk}[\Pi((x-x_{j})^3)i_{\alpha}] = F_{jk}[\Pi((x-x_{jk})^3)i_{\alpha}] - \frac{1}{4}(x_k - x_j)^3 (\bm{A}\cdot \bm{n}_{jk}\,i_{\alpha} ).
$$
Therefore,
$$
S_2 = \sum\limits_{j \in \mathcal{N}^1} \sum\limits_{k \in N(j)} F_{jk}[\Pi((x-x_{jk})^3)i_{\alpha}]  - \frac{1}{4} \bm{A}i_{\alpha} \cdot 
\sum\limits_{j \in \mathcal{N}^1} \sum\limits_{k \in N(j)} (x_k - x_j)^3  \bm{n}_{j,k}.
$$
The first term on the right-hand side is zero by $F_{jk}[u] = -F_{kj}[u]$. By Lemma~\ref{th:lemma14} we get
$$
S_2 = 3(d+1)(\bm{A}i_{\alpha}\cdot\B{e})\sum\limits_{v \in \mathcal{E}^0} M_{v,\bm{e}}.
$$
It remains to see that $S_1 + S_2 = 0$.
\end{proof}

\begin{corollary}\label{th:FC_is_3exact}
For each $d \in \mathbb{N}$, the flux correction method is 3-exact on a simplicial TI-mesh.
\end{corollary}
\begin{proof}
On a TI-mesh, the truncation error on a third order polynomial is the same for each $j \in \mathcal{N}$. Since it equals zero in average, it equals zero for each $j$.
\end{proof}

\section{Convergence rate}
\label{sect:theory}

Consider a scheme of the form \eqref{eq003}--\eqref{eq004} for the Cauchy problem \eqref{eq001}--\eqref{eq002}. Let it be $p$-exact but not $(p+1)$-exact. In this section we show that zero mean error on $(p+1)$-th order polynomials:
\begin{itemize}
\item is necessary for the supra-convergence;
\item under additional assumptions, allows to isolate the divergence term in the truncation error and thus to prove the $(p+1)$-th order convergence.
\end{itemize}

\subsection{Notation}

We say that a mesh has period $L>0$ if it is invariant with respect to the translation on each vector $zL$, $z \in \mathbb{Z}^d$. Then for a mesh function the translation in $L$ along each coordinate axis is naturally defined. We say that a mesh function $f \in \mathbb{C}^{\mathcal{M}}$ has period $L$ and write $f \in V_{per}^L$ if it is invariant with respect to the translation in $L$ along each axis. 
Obviously, $V_{per}^L \subset V_{per}^1$ if $1/L \in \mathbb{N}$. Denote $V_{per} \equiv V_{per}^1$. For $f\in V_{per}$ and $g \in (V_{per})^n$ we use the norms
\begin{equation}\label{nnorms}
\|f\|^2 = \sum\limits_{j \in \mathcal{M}^1} |K_j|\ |f_j|^2, \quad 
\|g\|^2 = \sum\limits_{j \in \mathcal{M}^1} |K_j|\ \|g_j\|^2,
\end{equation}
and the Euclidean norm is used on $\mathbb{C}^n$.

For $f \in (V_{per})^n$, let $u(t, f)$ be the solution of \eqref{eq003} such that $u(0,f) = f$ and $u(t,f) \in (V_{per})^n$ for each $t \ge 0$. The stability function for a mesh is 
$$
K(t) = \sup\limits_{0\le \tau \le t}\ \sup\limits_{f \in V_{per} \setminus \{0\}} \frac{\|u(\tau, f)\|}{\|f\|}.
$$
A scheme is called stable on a set of meshes if $K(t)$ is bounded on this set for each $t \ge 0$. 

For $f \in (C^1(\mathbb{R}^d))^n$, the truncation error $\epsilon(f,\Pi)$ is defined by \eqref{def00_apprerr}. Since the scheme is $p$-exact, then $\epsilon(f, \Pi) \in (V_{per}^L)^n$ for each $f \in (P_{p+1})^n$ and each mesh period $L$. 

Denote $u(t) = u(t, \Pi w_0)$. Let $w(t,\bm{r})$ be the solution of \eqref{eq001}--\eqref{eq002}. The solution error on a 1-periodic mesh is  
\begin{equation}
\varepsilon(t) = u(t) - \Pi w(t, \ \cdot\ ) \in (V_{per})^n.
\label{eq_def_solerr}
\end{equation}


For $j \in \mathcal{M}$ and $k \in \mathcal{S}(j)$, let $a_{jk}$ and $m_{jk}$ be the $(n \times n)$-matrices of the coefficients in \eqref{eq003}. Let $A$ and $M$ be the operators 
taking $u \in (V_{per})^n$ to $Au \in (V_{per})^n$ and $Mu \in (V_{per})^n$ with the components
\begin{equation}
({A} u)_j = \sum\limits_{k \in \mathcal{S}(j)} a_{jk} u_k, 
\quad
({M} u)_j = \sum\limits_{k \in \mathcal{S}(j)} m_{jk} u_k, 
\quad j \in \mathcal{M}.
\label{eq_def_breve}
\end{equation}

We naturally assume that the mesh scaling with a coefficient $h$ does not change $m_{jk}$ and multiplies $a_{jk}$ by $h^{-1}$.

\begin{assumption}\label{ass:3}
Operator $A$ has the form
$$
({A}u)_j = \frac{1}{|K_j|} \sum\limits_{k \in N(j)} F_{jk}[u], \quad F_{jk}[u] = -F_{kj}[u].
$$
\end{assumption}

For the finite-volume methods defined in Section~\ref{sect:schemes}, Assumption~\ref{ass:3} is by construction. It is not equivalent to the conservation, because it ignores the coefficients $m_{jk}$.

\begin{lemma}\label{th:19anjad}
For each $u \in (V_{per})^n$ there holds 
\begin{equation}
\sum\limits_{j \in \mathcal{M}^1} |K_j|\ ({A}u)_j = 0.
\label{eq_ass2}
\end{equation}
\end{lemma}
\begin{proof}
Let $u \in (V_{per})^n$. Then for each $m \in \mathbb{N}$ there holds
$$
\sum\limits_{j \in \mathcal{M}^1} |K_j|\ ({A}u)_j = \frac{1}{m^d} \sum\limits_{j \in \mathcal{M}^m} |K_j|\ ({A}u)_j = 
\frac{1}{m^d} \sum\limits_{j \in \mathcal{M}^m} \sum\limits_{k \in N(j)} F_{jk}[u].
$$
The fluxes between $j,k \in \mathcal{M}^m$ negate each other, and only fluxes between $j \in \mathcal{M}^m$ and $k \not\in \mathcal{M}^m$ remain. Each flux is $O(1)$, and the number of these fluxes is $O(m^{d-1})$ as $m \rightarrow \infty$. Taking $m$ to infinity we obtain \eqref{eq_ass2}.
\end{proof}

Equation \eqref{eq_ass2} is equivalent to 
$$
\sum\limits_{j\,:\,k \in \mathcal{S}(j)} |K_j|\ a_{jk} = 0, \quad k \in \mathcal{M}.
$$

By $\bm{m} = (m_1, \ldots , m_d)$ denote the multiindex $m_i \geqslant 0$, $|\bm{m}| = m_1+\ldots+m_d$, $\bm{m}!~=~m_1! \ldots m_d!$. 
For $\bm{r} = (x_1, \ldots, x_d)$ and a multiindex $\bm{m}$ denote
\begin{equation}
\bm{r}^{\bm{m}} = x_1^{m_1} \ldots x_d^{m_d}, 
\quad
D^{\bm{m}} = \frac{\D^{|\bm{m}|}}{\D x_1^{m_1} \ldots \D x_d^{m_d}}.
\label{eq_defgradnorm1}
\end{equation}

For $f = \{f_{\alpha} \in C^{k}(\mathbb{R}^d), \alpha = 1, \ldots, n\}$, $k \in \mathbb{N}$, denote
$$
|f|_{k,\infty} = \max\limits_{\alpha=1, \ldots, n}\ \sup\limits_{\bm{r} \in \mathbb{R}^d}  \max\limits_{|\bm{m}|=k}\ |D^{\bm{m}} f_{\alpha}(\bm{r})|.
$$

\subsection{Criterion of the supra-convergence for the transport equation}

Consider the case $n=1$. Then \eqref{eq001} is the transport equation $w_t + \bm{\omega} \cdot \nabla w = 0$ with some velocity $\bm{\omega} \in \mathbb{R}^d$. For this case, Theorems~3.1 and~10.5 in~\cite{Bakhvalov2023} give criteria of the $(p+1)$-th order convergence if the mesh refinement is by scaling. Here we present a corollary of these theorems in our notation.

\begin{proposition}\label{th:MT1}
Let $n=1$. Consider a 1-periodic mesh $X$. For $h$ such that $1/h \in \mathbb{N}$, let $X_h$ be the mesh generated by scaling $X$ with the coefficient $h$. On $\{X_h\}$, consider a $p$-exact stable scheme of the form \eqref{eq003}--\eqref{eq004}.

1. If the truncation error on mesh $X$ satisfies
\begin{equation}
\epsilon(f, \Pi) \in \mathrm{Im} A \quad \mathrm{for\ each}\quad f \in P_{p+1},
\label{eq_ihj}
\end{equation}
then the solution error on each $X_h$ satisfies
$$
\|\varepsilon(t)\| \le C h^{p+1} (|w_0|_{p+1,\infty} + t |w_0|_{p+2,\infty})
$$
with some $C > 0$ independent of $h$. 

2. If \eqref{eq_ihj} does not hold, then there exists $\bm{\alpha} \in \mathbb{Z}^d$ such that for $w_0(\bm{r}) = \exp(2\pi \bm{\alpha} \cdot \bm{r})$ each estimate of the form 
$$
\|\varepsilon(t)\| \le C h^{p+\delta} (1+t), \quad \delta > 0, \quad C > 0,
$$
does not hold for some $h$ (such that $1/h \in \mathbb{N}$) and $t > 0$.
\end{proposition}

By Lemma~\ref{th:19anjad}, $\mathrm{Im} {A} \subset \{f \in V_{per}: \sum_j |K_j|\, f_j = 0\}$. Hence, the zero mean error condition \eqref{def00_zeroav} is necessary for \eqref{eq_ihj}. 
Combining this with the second statement of Proposition~\ref{th:MT1} we come to the conclusion that for finite-volume schemes, the zero mean error on $(p+1)$-order polynomials is a {\bf necessary condition of the supra-convergence}, at least for the scalar case and refinenent by scaling. 


The following simple result is useful for the analysis.

\begin{lemma}\label{th:lemma_very_aux}
Let $L$, $1/L \in \mathbb{N}$, be a mesh period and $\breve{A}$ be the restriction of $A$ to $V_{per}^L$. Then $\mathrm{Im} \breve{A} = \mathrm{Im} A \cap V_{per}^L$.
\end{lemma}
\begin{proof}
Denote $a = 1/L$. Let $f \in \mathrm{Im} A \cap V_{per}^L$. Then $f = Ag$ for some $g \in V_{per}$. Let $T$ be the operator of translation by $L$. Then
$$
f = \frac{1}{a} \sum\limits_{j=0}^{a-1} T^j f = A \tilde{g}, \quad \tilde{g} = \frac{1}{a} \sum\limits_{j=0}^{a-1} T^j g \in V_{per}^L.
$$
Thus, $f \in \mathrm{Im} \breve{A}$. The reverse embedding is obvious.
\end{proof}



\subsection{Sufficient condition for supra-convergence}

The goal of this subsection is to get an estimate of the solution error that reveals the supra-convergence. This generalizes the first statement in Proposition~\ref{th:MT1} to linear systems. 

Below for a mesh we use its step $h$. For instance, $h$ may be defined as the maximal edge length. The purpose of this notation is that the ``constants'' $C_m$, $C_W$, $C_a$, $C_{\tilde{a}}$, $C_{\Pi}$, $C_E$ (defined below) are invariant with respect to mesh scaling.

\begin{assumption}\label{ass:4}
The scheme is $p$-exact (for some $p \in \mathbb{N}$), and the scheme has zero mean error on the $(p+1)$-th order polynomials.
\end{assumption}

For the schemes defined in Section~\ref{sect:schemes}, this assumption was verified in Sections~\ref{sect:check1} and~\ref{sect:check2}.

\begin{assumption}\label{ass:2}
For the minimal mesh period $L$, each $u \in (V_{per}^L)^n \cap \mathrm{Ker}\,A$ is constant in space, i. e. $u_j^{\alpha} = u_k^{\alpha}$ for each $j, k \in \mathcal{M}$ and each $\alpha = 1, \ldots, n$.
\end{assumption}

If the refinement by scaling is used, then \mbox{$L/h = const$}, and Assumption~\ref{ass:2} restricts the scheme behavior on the oscillations with the wavelength $O(h)$. If all these oscillations are damped at a time $O(h)$ (which is the Giles' assumption, see Section~\ref{sect:history}), then Assumption~\ref{ass:2} holds. If there exists a nonconstant steady numerical solution with a period $O(h)$, then Assumption~\ref{ass:2} does not hold.


Denote
\begin{equation}
C_A = h^{-1} \sup\limits_{b \in (V_{per}^L)^n \cap \mathrm{Im}\,A}\ \ \inf\limits_{x\in (V_{per}^L)^n \,:\,Ax=b} \frac{\|x\|}{\|b\|}.
\label{eq_def_CA}
\end{equation}
Since $a_{jk} \sim h^{-1}$ when mesh scaling, then $C_A$ is invariant with respect to the mesh scaling. If one uses the natural scalar product corresponding to the norms \eqref{nnorms} in use, then $C_A$ is just the reciprocal of the minimal positive singular value of the restriction of $hA$ to $(V_{per}^L)^n$. The value of $C_A$  is a measure of how close the scheme is to fail Assumption~\ref{ass:2}.

Our main result is as follows.

\begin{theorem}\label{th:main14}
Let the scheme \eqref{eq003} satisfy Assumptions~\ref{ass:3}--\ref{ass:2}. Let $w_0 \in C^{p+2}(\mathbb{R}^d)$, $w(t,\bm{r})$ be the solution of \eqref{eq001}--\eqref{eq002}, and $u(t)$ be the solution of \eqref{eq003}--\eqref{eq004}. Then the solution error $\varepsilon(t) = u(t) - \Pi w(t, \ \cdot\ ) \in (V_{per})^n$ satisfies
\begin{equation}
\label{eq_mainth}
\begin{gathered}
\|\varepsilon(t)\| \le C_{\Pi} (K(t) |w(t,\ \cdot\ )|_{p+1,\infty} + |w_0|_{p+1,\infty}) h^{p+1}  +  
\\
+ C_E K(t) h^{p+1} t \sup\limits_{0 \le \tau \le t} |w(\tau, \ \cdot\ )|_{p+2,\infty}.
\end{gathered}
\end{equation}
The constants $C_{\Pi}$ and $C_E$ are defined below. They do not depend on $w_0$ and are invariant with respect to mesh scaling.
\end{theorem}

The idea of the proof is to represent the truncation error on a function $f$ in the form
\begin{equation}
\epsilon_j(f, \Pi) = - \sum\limits_{k \in \mathcal{S}(j)} a_{jk} \xi_k[f] + \eta_j[f], \quad \xi_k[f] = \sum\limits_{|\bm{m}|=p+1} \sum\limits_{\alpha=1}^n \mathfrak{C}_k^{(\bm{m},\alpha)} D^{\bm{m}} f_{\alpha}(\bm{r}_k)
\label{eq_86}
\end{equation}
with some $\mathfrak{C}_k^{(\bm{m},\alpha)} \in \mathbb{R}^n$.



Denote
$$
C_{\epsilon} = h^{-p} \max\limits_{|\bm{m}|=p+1} \|\epsilon(\bm{r}^{\bm{m}}/\bm{m}!, \Pi)\|.
$$
Obviously, $C_{\epsilon}$ is invariant with respect to the mesh scaling.

\begin{lemma}\label{th:lemma7}
There exist $n$-vectors $\mathfrak{C}_j^{(\bm{m},\alpha)}$, $|\bm{m}|=p+1$, $\alpha=1,\ldots,n$, periodic in $j$ with each mesh period such that \eqref{eq_86} holds with $\eta_j[f] = 0$ for each \mbox{$f \in (P_{p+1})^n$}. Besides, for $\mathfrak{C}^{(\bm{m},\alpha)} = \{\mathfrak{C}_j^{(\bm{m},\alpha)}, j \in \mathcal{M}\} \in (V_{per})^n$ there holds
$$
\|\mathfrak{C}^{(\bm{m},\alpha)}\| \le h^{p+1} C_A C_{\epsilon}.
$$
\end{lemma}
\begin{proof} 
Let $i_{\alpha}$, $\alpha = 1, \ldots, n$, be the vectors of the standard basis in $\mathbb{C}^n$. Since the scheme is $p$-exact, it is sufficient to consider the functions $f_{\alpha, \bm{m}} = i_{\alpha} \bm{r}^{\bm{m}}/\bm{m}!$, $|\bm{m}|=p+1$, $\alpha=1,\ldots,n$. For each of them, the condition $\eta(f_{\alpha, \bm{m}}, \Pi) = 0$ reduces to the system
$$
-\epsilon_j(i_{\alpha}\bm{r}^{\bm{m}}/\bm{m}!, \Pi) = 
\sum\limits_{k \in \mathcal{S}(j)} a_{jk} \mathfrak{C}_k^{(\bm{m}, \alpha)}, \quad j \in \mathcal{M},
$$
with respect to $\{\mathfrak{C}_j^{(\bm{m}, \alpha)}, j \in \mathcal{M}\}$. Or, in the matrix, form,
\begin{equation}
A\mathfrak{C}^{(\bm{m},\alpha)} = - \epsilon(i_{\alpha}\bm{r}^{\bm{m}}/\bm{m}!, \Pi).
\label{eq_aidohao}
\end{equation}

By Assumption~\ref{ass:2}, the kernel of $A$ has dimension $n$ and consists of the vectors $e^{(\beta)}$ with the components $(e^{(\beta)})_j^{\alpha} = \delta_{\alpha \beta}$, where $\delta$ is the Kronecker symbol. By Assumption~\ref{ass:3}, the vectors $\tilde{e}^{(\beta)}$ with the components $(\tilde{e}^{(\beta)})_j^{\alpha} = |K_j|\, \delta_{\alpha \beta}$ belong the the co-kernel of $A$. Since the dimension of the co-kernel is $n$, then $\tilde{e}^{(\beta)}$ form the basis in the co-kernel. Thus, system \eqref{eq_aidohao} is consistent iff
$$
\sum\limits_{j \in \mathcal{M}^1} |K_j|\ \epsilon_j(i_{\alpha} \bm{r}^{\bm{m}}/\bm{m}!, \Pi) = 0, \quad \alpha = 1, \ldots, n.
$$
This is the zero mean error condition on $(p+1)$-order polynomials, and by Assumption~\ref{ass:4} it holds. The inequality to prove follows from \eqref{eq_def_CA}.
\end{proof}

\begin{lemma}\label{th:lemma8}
Let $f \in (C^{p+1}(\mathbb{R}^d))^n$ be 1-periodic, $\xi_k[f]$ be defined by \eqref{eq_86} with $\mathfrak{C}_k^{(\bm{m}, \alpha)}$ given by Lemma~\ref{th:lemma7}, and \mbox{$\xi[f] = \{\xi_k[f], k\in \mathcal{M}\}$}. Then 
$$
\|\xi[f]\| \le C_{\Pi} h^{p+1} |f|_{p+1,\infty}
$$
with
\begin{equation}\label{eq_def_Cpi}
C_{\Pi} = nc_p  C_A C_{\epsilon}, \quad c_p = (p+d)! / ((p+1)! (d-1)!).
\end{equation}
\end{lemma}
\begin{proof}
The sum in $\bm{m}$ in \eqref{eq_86} has $\dim P_{p+1} - \dim P_p = c_p$ terms. Then
$$
\|\xi_k[f]\| \le c_p n \max\limits_{|\bm{m}|=p+1} \max\limits_{\alpha=1, \ldots, n} \|\mathfrak{C}_k^{(\bm{m},\alpha)}\|\ |f|_{p+1,\infty}
$$
and
$$
\|\xi[f]\| \le c_p n \max\limits_{|\bm{m}|=p+1} \max\limits_{\alpha=1, \ldots, n} \|\mathfrak{C}^{(\bm{m},\alpha)}\|\ |f|_{p+1,\infty}.
$$
It remains to use the estimate for $\|\mathfrak{C}_{j}^{(\bm{m},\alpha)}\|$ given by Lemma~\ref{th:lemma7}.
\end{proof}

Denote
$$
C_W = h^{-1} \max\limits_{j \in \mathcal{M}} \max\limits_{k \in \mathcal{S}(j)} |\bm{r}_k - \bm{r}_j|,
$$
$$
C_m = \max\limits_{j \in \mathcal{M}}\sum\limits_{k \in \mathcal{S}(j)} \|m_{jk}\|, \quad
C_a = h\ \max\limits_{j \in \mathcal{M}} \sum\limits_{k \in \mathcal{S}(j)} \|a_{jk}\|, \quad
C_v = \frac{\max\limits_{j \in \mathcal{M}} |K_j|}{\min\limits_{j \in \mathcal{M}} |K_j|},
$$
$$
\tilde{C}_a = h\ \max\limits_{j \in \mathcal{M}} \max\limits_{k \in \mathcal{S}(j)} \|a_{jk}\|\ \left(\max\limits_{j \in \mathcal{M}} |\mathcal{S}(j)|\right)^{1/2} \left( \max\limits_{k \in \mathcal{M}} |\{j\ :\ k \in \mathcal{S}(j)\}|\right)^{1/2} C_v^{1/2},
$$
$$
\|\bm{A}\| = \sup\limits_{\bm{e} \in \mathbb{R}^d, \ |\bm{e}|=1} \|\bm{A} \cdot \bm{e}\|
$$
Finally, put
\begin{equation}
C_E = \|M^{-1}\|(\sqrt{d} \|\bm{A}\|\,C_m C_W^{p+1} + C_a C_W^{p+2} + nc_p  \tilde{C}_a C_W C_A C_{\epsilon}).
\label{eq_def_CE}
\end{equation}

\begin{lemma}\label{th:lemma9}
In the notation of Lemma~\ref{th:lemma7}, for a 1-periodic function \mbox{$f \in (C^{p+2}(\mathbb{R}^d))^n$} and for $\eta[f] = \{\eta_j[f], j \in \mathcal{M}\} \in (V_{per})^n$ there holds
$$
\|M^{-1} \eta[f]\| \le h^{p+1} C_E |f|_{p+2,\infty}.
$$
\end{lemma}
\begin{proof}
Let $j \in \mathcal{M}$. Let $\tilde{f}^{(j)}(\bm{r})$ be the Taylor polynomial of the function $f$ of order $p+1$ at $\bm{r}_j$. Put $g^{(j)} = f - \tilde{f}^{(j)}$. Since  $\eta_j[\tilde{f}^{(j)}] = 0$, then
$$
\eta_j[f] = \eta_j[g^{(j)}] = \psi_j + \sum\limits_{|\bm{m}|=p+1} \sum\limits_{\alpha=1}^n \phi_j^{(\bm{m},\alpha)},
$$
$$
\psi_j = \epsilon_j(g^{(j)}, \Pi), \quad \phi_j^{(\bm{m},\alpha)} = \sum\limits_{k \in \mathcal{S}(j)} a_{jk} \mathfrak{C}_k^{(\bm{m}, \alpha)} D^{\bm{m}} g^{(j)}_{\alpha}(\bm{r}_k).
$$
Since
$$
\|g^{(j)}(\bm{r})\| \le |f|_{p+2,\infty} |\bm{r}-\bm{r}_j|^{p+2},  \quad
\|\bm{A} \cdot \nabla g(\bm{r})\| \le \sqrt{d} \|\bm{A}\|\, |f|_{p+2,\infty} |\bm{r}-\bm{r}_j|^{p+1},
$$
then
\begin{equation}
\|\psi_j\| \le h^{p+1} (\sqrt{d}\|\bm{A}\|\,C_m C_W^{p+1} + C_a C_W^{p+2}) |f|_{p+2,\infty}.
\label{eq_aux_1i2jpo}
\end{equation}
Hence, $\|\psi\|$ is bounded by the right-hand side of \eqref{eq_aux_1i2jpo}.

Now let $|\bm{m}| = p+1$ and $\alpha \in (1, \ldots, n)$. We have
$$
\|\phi_j^{(\bm{m},\alpha)}\| \le \sum\limits_{k \in \mathcal{S}_j} \|\mathfrak{C}_k^{(\bm{m},\alpha)}\|\ \max\limits_{j \in \mathcal{M}} \max\limits_{k \in \mathcal{S}(j)} \|a_{jk}\| hC_W |f|_{p+2,\infty}.
$$
Since
$$
\sum\limits_{j \in \mathcal{M}^1} |K_j| \left(\sum\limits_{k \in \mathcal{S}(j)} \|\mathfrak{C}_k^{(\bm{m},\alpha)}\|\right)^2 \le \sum\limits_{j \in \mathcal{M}^1} |K_j| |\mathcal{S}_j| \sum\limits_{k \in \mathcal{S}(j)} \|\mathfrak{C}_k^{(\bm{m},\alpha)}\|^2 \le
$$
$$
\le
\|\mathfrak{C}^{(\bm{m},\alpha)}\|^2\ C_v\  \max\limits_{k \in \mathcal{M}} \sum\limits_{j : k\in \mathcal{S}(j)} |\mathcal{S}(j)|,
$$
we get
$$
\|\phi^{(\bm{m},\alpha)}\| \le \|\mathfrak{C}^{(\bm{m},\alpha)}\| C_v^{1/2} \left( \max\limits_{k \in \mathcal{M}} \sum\limits_{j : k\in \mathcal{S}(j)} |\mathcal{S}(j)|\right)^{1/2} \max\limits_{j \in \mathcal{M}} \max\limits_{k \in \mathcal{S}(j)} \|a_{jk}\| hC_W |f|_{p+2,\infty} \le 
$$
$$
\le \|\mathfrak{C}^{\bm{m},\alpha}\| C_{\tilde{a}} C_w |f|_{p+2,\infty},
$$
and the estimate to prove easily follows.
\end{proof}

\begin{lemma}\label{th:errest}
Let $\mathcal{B}$ be a Banach space, $A$ and $M$ be continuous operators in $\mathcal{B}$, $\|{M}^{-1}\| < \infty$, $\xi \in C^1(\mathcal{B})$, $\eta \in C(\mathcal{B})$. Let $\|\exp(-{M}^{-1}{A} t)\| \le K(t)$  hold for each $t \ge 0$ and some non-decreating $K(t)$. Then the solution of
\begin{equation}\label{eq_abstract_eq}
{M} \frac{du}{dt} + {A} u = {A} \xi(t) + \eta(t), \quad u(0) = u_0 \in \mathcal{B}.
\end{equation}
satisfies
$$
\|u(t)\| \le t K(t) \sup\limits_{0 \le \tau \le t} (\|{M}^{-1} \eta(t)\| + \|\xi'(t)\|) + \|\xi(t)\| + K(t) \|u_0 - \xi(0)\|.
$$
\end{lemma}
\begin{proof}
The solution of \eqref{eq_abstract_eq} is
\begin{equation*}
\begin{gathered}
u(t) = \exp(-t{M}^{-1}{A}) u_0 + \int\limits_{0}^t \exp(-(t-\tau) {M}^{-1} {A}) ({M}^{-1} {A} \xi(\tau) + {M}^{-1} \eta(\tau))  d\tau =
\\
= \exp(-t{M}^{-1}{A}) u_0 + \xi(t) - \exp(-t{M}^{-1}{A}) \xi(0) + 
\\
+
\int\limits_0^t  \exp(-(t-\tau){M}^{-1}{A}) ({M}^{-1}\eta(t) - \xi'(t)) d\tau.
\end{gathered}
\end{equation*}
From here, the statement of the lemma is obvious.
\end{proof}

\begin{proof}[Proof of Theorem~\ref{th:main14}]
Taking \eqref{eq_def_solerr} to \eqref{eq003}--\eqref{eq004} we get
$$
\sum\limits_{k \in \mathcal{S}(j)} m_{jk} \frac{d\varepsilon_k}{dt} + \sum\limits_{k \in \mathcal{S}(j)} a_{jk} \varepsilon_k = \sum\limits_{k \in \mathcal{S}(j)} a_{jk} \xi_{k}[f] - \eta_k[f], \quad j \in \mathcal{M},
$$
$$
\varepsilon_j(0) = 0, \quad j \in \mathcal{M}.
$$
By Lemma~\ref{th:errest} we get
$$
\|\varepsilon(t)\| \le t K(t) \sup\limits_{0 \le \tau \le t} \left(\|M^{-1}\eta[w(\tau, \ \cdot\ )\| + \left\|\xi\!\left[\frac{\D w}{\D t}(\tau, \ \cdot\ )\right]\right\|\right) + 
$$
$$
+ \|\xi[w(t,\ \cdot\ )\| + K(t) \|\xi[w_0]\|.
$$
The use of the estimates for $\|\xi[w]\|$ and $\|\xi[\D w/\D t]\|$ given by Lemma~\ref{th:lemma8} and the estimate for $\|M^{-1}\eta[w]\|$ given by Lemma~\ref{th:lemma9} concludes the proof.
\end{proof}

\subsection{Discussion}

Theorem~\ref{th:main14} states the $(p+1)$-th order convergence under several assumptions. Consider them in the case of the transport equation. We use Proposition~\ref{th:MT1} as the criterion of supra-convergence.

i. Assumption~\ref{ass:2}: there should be no non-constant vectors in $(V_{per}^L)^n \cap \mathrm{Ker}\,A$. For the transport equation with the unit velocity, consider the standard P1 mass-lumped Galerkin method
$$
\frac{du_j}{dt} + \frac{1}{2\newhbar_{j}}(u_{j+1}-u_{j-1}) = 0,
$$
where $\newhbar_j = (x_{j+1}-x_{j-1})/2$. Consider the checkerboard mesh: $x_{2k} = 2kh$, \mbox{$x_{2k+1} = 2kh+0.5$} where \mbox{$1/(2h) \in \mathbb{N}$}. Its minimal period is $L = 2h$. Let $\Pi$ be the pointwise mapping.
It is easy to verify that this scheme is not 2-exact. The mean truncation error on second-order polynomials is zero (see Proposition~\ref{th:prop:EB}). However, the restriction of $A$ to $V_{per}^L$ is $\breve{A} = 0$. Thus, $\epsilon(x^2, \Pi) \not \in \mathrm{Im} \breve{A}$. By Lemma~\ref{th:lemma_very_aux}, $\epsilon(x^2, \Pi) \not \in \mathrm{Im} A$, and by Proposition~\ref{th:MT1}, there is no second-order convergence. These considerations are also valid for the P1 Galerkin method with no mass lumping. Note that the second order-accurate solution may be obtained by post-processing \cite{Cockburn2003}.


ii. The constant $C_A$ is used via $C_{\Pi}$ and $C_E$ in \eqref{eq_mainth}. For most of the schemes on unstructured meshes ($d>1$), it is hardly possible to estimate $C_A$. All we know is that $C_A$ is invariant with respect to the mesh scaling. It may happen that $C_A$ grows with the number of degrees of freedom per mesh period, and a non-integer convergence order is possible \cite{Peterson1991, Bakhvalov2023b}. 

iii. Assumption~\ref{ass:3} may be considered as the definition of the weights $|K_j|$ that should be used in the definition \eqref{def00_zeroav_0} of the mean truncation error. 

iv. Assumption~\ref{ass:4}: zero mean error on the $(p+1)$-order polynomials. One may suggest that $(p+1)$-exactness on uniform meshes is enough. The example in the introduction shows that this suggestion is wrong. Consider the transport equation $w_t + w_x = 0$. Let $D$ and $L$ be the operators taking a mesh function \mbox{$u = \{u_j, j \in \mathbb{Z}\}$} to the first and the second derivatives, correspondingly, at $x_j$ of the Lagrange polynomial based on $x_{j-1}$, $x_j$, $x_{j+1}$. The flux correction method with the ``extended Galerkin formula'' has the form \eqref{eq_intro_1}. In Section~\ref{sect:FC_strict} we proved that it has zero mean error on the third order polynomials. Now consider the following scheme
\begin{equation*}
\newhbar_j \frac{du_j}{dt} - \newhbar_j\frac{h_{j+1/2}^2 + h_{j-1/2}^2}{24} \frac{d(Lu)_j}{dt} + u_j + \frac{1}{2} h_{j+1/2} (Du)_j - u_{j-1} - \frac{1}{2} h_{j-1/2} (Du)_{j-1} = 0.
\end{equation*}
It is 2-exact a non-uniform mesh and 3-exact a uniform mesh. The difference between the average trunction error on the function $4x^3$ between this schemes is
\begin{equation*}
\begin{gathered}
\sum\limits_{j \in \mathcal{M}^1} \newhbar_j \epsilon_j(4x^3,\Pi) =
\sum\limits_{j \in \mathcal{M}^1} \left(\newhbar_j (h_{j+1/2}^2+h_{j-1/2}^2) - (h_{j+1/2}^3 + h_{j-1/2}^3)\right) =
\\
= - \sum\limits_{j \in \mathcal{M}^1} \newhbar_j (h_{j+1/2}^2 - h_{j-1/2}^2)^2.
\end{gathered}
\end{equation*}
This is nonzero for any non-uniform mesh. By Proposition~\ref{th:MT1}, the numerical solution does not converge with the order $2+\delta$ for each $\delta>0$.

v. Error estimate \eqref{eq_mainth} explicitly depends on the stability function. In contrast to finite-element methods, there are no theoretical results establishing the stability of high-order finite-volume methods on unstructured meshes. When using these methods, one usually relies on the numerical dissipation caused by the use of upwind fluxes.\\

The zero-mean-error condition can be naturally extended to finite-element methods. In particular, the standard Galerkin method is constructed such that the truncation error (i. e. what we call a truncation error in the finite-difference methods) is orthogonal to the space of test functions. The orthogonality to the function $f=1$ constitutes the zero mean error condition. However, for the accuracy analysis of finite-element methods, there are more powerful methods than what we consider here. Thus, we do not consider finite-element methods in this paper.

\section{Numerical results}
\label{sect:numresults}

Each paper that proposes a new scheme contains verification data, so there are extensive results demonstrating the $(p+1)$-th order convergence of the schemes we consider. Among other papers, the reader may find them in \cite{Tsoutsanis2011, Tsoutsanis2018, Tsoutsanis2021} for schemes with a polynomial reconstruction, in \cite{Abalakin2016} for 1-exact edge-based schemes, in \cite{Pincock2014,Nishikawa2017} for the flux correction method.
It seems redundant to repeat those results here. So we concentrate on a fact that was not published before, which is the third order convergence of the multislope cell-centered scheme BBR3 (see Section~\ref{sect:def_multislope}) on a translationally-invariant triangular mesh for steady problems. To emphasize the significance of this result, we also apply this scheme to a steady flow around NACA0012 body.

\subsection{Linear vortex on a regular-triangular mesh}

Consider the 2D Euler equations linearized on the steady uniform field $\bar{\rho}=1$, $\bar{\bm{u}} = (\bar{u}, \bar{v})^T$, $\bar{p}=1/\gamma$ where $\gamma$ is the specific ratio. They have the form
\begin{equation*}
\begin{gathered}
 \frac{\D w}{\D t} + A_x \frac{\D w}{\D x} + A_y \frac{\D w}{\D y}  = 0,\\
A_x = \left(\begin{array}{lccr} \bar{u} & 1 & 0 & 0 \\ 0 & \bar{u} & 0 & 1 \\
 0 & 0 & \bar{u} & 0 \\ 0 & 1 & 0 & \bar{u}
 \end{array}\right),\
 A_y = \left(\begin{array}{lccr} \bar{v} & 0 & 1 & 0 \\ 0 & \bar{v} & 0 & 0 \\
 0 & 0 & \bar{v} & 1 \\ 0 & 0 & 1 & \bar{v}
 \end{array}\right)
 \end{gathered}
\end{equation*}
where $w = (\rho', \bm{u}', p')^T$ containts the pulsations of density, velocity, and pressure.

Let the initial data prescribe the vortex
$$
\rho' = p' = 0, \quad \bm{u}' = \nabla^{\perp} \psi
$$
with $\psi = \exp(-\bm{r}^2/800)$. We consider the steady case $\bar{\bm{u}} = 0$ and the unsteady case $\bar{\bm{u}} = (0.4,0)^T$. We use the translationally-invariant triangular meshes with $\bm{e}_1 = (h,0)^T$, $\bm{e}_2 = (h/2, 5h/6)^T$ (see Fig.~\ref{fig:bbr_regular}). At the time moment $t = 1$ we evaluate the norm of the solution error
$$
\|\varepsilon(t)\| = \left(\sum\limits_{j \in \mathcal{M}^1} |K_j|\, \|u_j(t) - w(t, \bm{r}_j)\|^2 \right)^{1/2}.
$$

The results for $\|\varepsilon(1)\|$ are given in Table~\ref{table:regulartri}. The numerical solution for the steady problem converges with the third order, and the numerical convergence order becomes close to 2 as the mesh is refined.

\begin{table}[t]
\caption{\label{table:regulartri}Solution error of the BBR3 scheme on regular triangular meshes for the linear vortex case}
\begin{center}
\begin{tabular}{|c|c|c|c|c|}
\hline
$h$ & $\bar{\bm{u}}=(0.4,0)^T$ & order & $\bar{\bm{u}}=0$ & order  \\
\hline
$1/20$ & $4.18 \cdot 10^{-1}$ & & $2.72 \cdot 10^{-1}$ & \\
$1/40$ & $7.12 \cdot 10^{-2}$ & 2.55 & $3.77 \cdot 10^{-2}$ & 2.85 \\
$1/80$ & $1.20 \cdot 10^{-2}$ & 2.57 & $4.60 \cdot 10^{-3}$ & 3.04 \\
$1/160$ & $2.43 \cdot 10^{-3}$ & 2.30 & $5.48 \cdot 10^{-4}$ & 3.07 \\
$1/320$ & $5.61 \cdot 10^{-4}$ & 2.12 & $6.63 \cdot 10^{-5}$ & 3.05 \\
\hline
\end{tabular}
\end{center}
\end{table}

\subsection{Multislope method on prismatic layers}

A more computationally relevant case is a high-Reynolds number flow around a body. We use the Reynolds-averaged Navier~-- Stokes equation with the Spalart~-- Allmaras turbulence model. They have the form
\begin{gather*}
\frac{\D Q(U)}{\D t} + \nabla \cdot \mathcal{F}(U, \nabla U) = 0,\\
\frac{\partial (\rho\tilde{\nu})}{\partial t} + \nabla \cdot(\rho \tilde{\nu} \mathbf{u}) = S_\nu(U,\nabla U,\tilde{\nu},\nabla \tilde{\nu}),
\end{gather*}
with $U = (\rho, \bm{u}, p)$, $Q = (\rho, \rho \bm{u}, E)^T$, and
$$
\mathcal{F} = \left(\begin{array}{c} \rho \bm{u} \\ \rho \bm{u}\otimes \bm{u} + p \bm{I} - \tau \\ (E+p) \bm{u} - \tau \cdot \bm{u} - \gamma(\mu/\mathrm{Pr}+\mu_T)\nabla (p/\rho) \end{array}\right).
$$
Here $\rho$ is the density, $\bm{u}$ is the velocity, $p$ is the pressure, $E = \rho |\bm{u}|^2/2 + p/(\gamma-1)$ is the total energy, $\gamma$ is the specific ratio, 
$$
\tau = (\mu + \mu_T) (\nabla \bm{u} + (\nabla \bm{u})^T - 2\B{I}\mathrm{div}\bm{u}/3)
$$
is the stress tensor, $\mu_T$ is the turbulent viscosity, $\rho \tilde{\nu}$ is the corrected turbulent viscosity, $\bm{I}$ is the unit tensor. The expression for $S_{\nu}$ and the relation between $\mu_T$ and $\tilde{\mu} = \rho \tilde{\nu}$ are given by the Spalart~-- Allmaras turbulence model \cite{Allmaras2012}.

A finite-volume scheme has the general form
$$
\frac{dQ_j}{dt} + \frac{1}{|K_j|} \sum\limits_{k \in N(j)} F_{jk}[U] = 0.
$$
To define the numerical fluxes for the main set of variables, we use the Roe solver applied to the variables reconstructed using the BBR3 scheme:
$$
F_{jk}[U] = \frac{\mathcal{F}(Q_{jk}) + \mathcal{F}(Q_{kj})}{2} \cdot \bm{n}_{jk} - \frac{1}{2} S_{jk}|\Lambda_{jk}|S_{jk}^{-1} (Q_{kj}-Q_{jk})
$$
where $Q_{jk}=Q(\mathcal{R}_{jk}[U])$, $Q_{kj}=Q(\mathcal{R}_{kj}[U])$, $U = \{U(Q_j), j \in \mathcal{E}\}$ is the set of physical values at cell centers, and $S_{jk}\equiv S_{kj}$ and $\Lambda_{jk} \equiv \Lambda_{kj}$ are the matrices of eigenvectors and eigenvalues of $d(\mathcal{F}\cdot \bm{n}_{jk})/dQ$ taken at the Roe average of $Q_j$ and $Q_k$. 

The multislope method described in Section~\ref{sect:def_multislope} is not directly applicable to anisotropic prismatic layers, so we make some adjustments. 

First, to define $\bm{n}_{jk}$, we use second-order meshes, i. e. each mesh edge is not a segment but a second-order curve. This reduces the oscillations at boundaries.

Second, for each element and each face, the points $\bm{r}_j$ and $\bm{r}_{jk}$ are defined as the arithmetic average of radius-vectors of its vertexes and not as the mass center.

Third, we define mesh layers. If a base of a prism belongs to a boundary with the no-slip boundary condition, let it belong to layer 0. If a prism has a common base with a prism of layer $l$, let it belong to layer $l+1$. Generally, the layers do not exhaust the mesh.

Consider elements $j$ and $k$ with a common face. If they do not belong to the same layer (or at least one of them does not belong to any layer), we use the algorithm as defined in Section~\ref{sect:def_multislope}. If $j$ and $k$ belong to a layer $l$, then the procedure is as follows. Let $j_{\pm}$ and $k_{\pm}$ be the prisms with the common base at layers $l \pm 1$. Put
$$
\bm{N}_{jk} = \bm{n}_{j_- j} + \bm{n}_{jj_+} + \bm{n}_{k_-k} + \bm{n}_{kk_+}
$$
where the undefined terms are dropped. Project the centers of all cells at layer $l$ that are neighboring to $j$ or $k$ to a plane orthogonal to $\bm{N}_{jk}$. Finally, define $\mathcal{R}_{jk}$ using the algorithm described in Section~\ref{sect:def_multislope} as if we had a 2D mesh. This resembles the construction of the EBR-PL scheme \cite{Bakhvalov2022b}.



\subsection{Flow around the NACA0012 body}

In this section, we consider the flow around the NACA0012 body \cite{NACA0012}. 

For high-order finite-volume schemes, it is difficult to preserve the accuracy near a boundary. Besides, the convergence rate deteriorates because the solution is not smooth at the trailing edge \cite{Wang2013}. Also the viscous terms and the convection of the turbulent viscosity are not approximated with a high order. So we are not able to show the high-order convergence in an integral norm.

Instead, we compare the results by the multislope scheme with the vertex-centered EBR3-PL and EBR5-PL schemes \cite{Bakhvalov2022b}. We use the computational meshes from that paper, an example of which is shown at Fig.~\ref{fig:meshNACA}. The brief information about the meshes is collected in Table~\ref{table:m1}.

\begin{table}[t]
\caption{\label{table:m1}Statistics of computational meshes}
\begin{center}
\begin{tabular}{|l|r|r|r|}
\hline
Mesh name & cx4 & cx3 & cx2  \\
\hline
Number of surface triangles & 12236 & 25440 & 57196  \\
\hline
Number of tetraherda & 339584 & 568364 & 1029976  \\
Number of pyramids & 6756 & 11408 & 22284  \\
Number of prisms & 1073612 & 2014152 & 3967356 \\
Number of hexahedra & 21052 & 21052 & 21052  \\
\hline
\end{tabular}
\end{center}
\end{table}

\begin{figure}[t]
\centering
\includegraphics[width=0.5\linewidth]{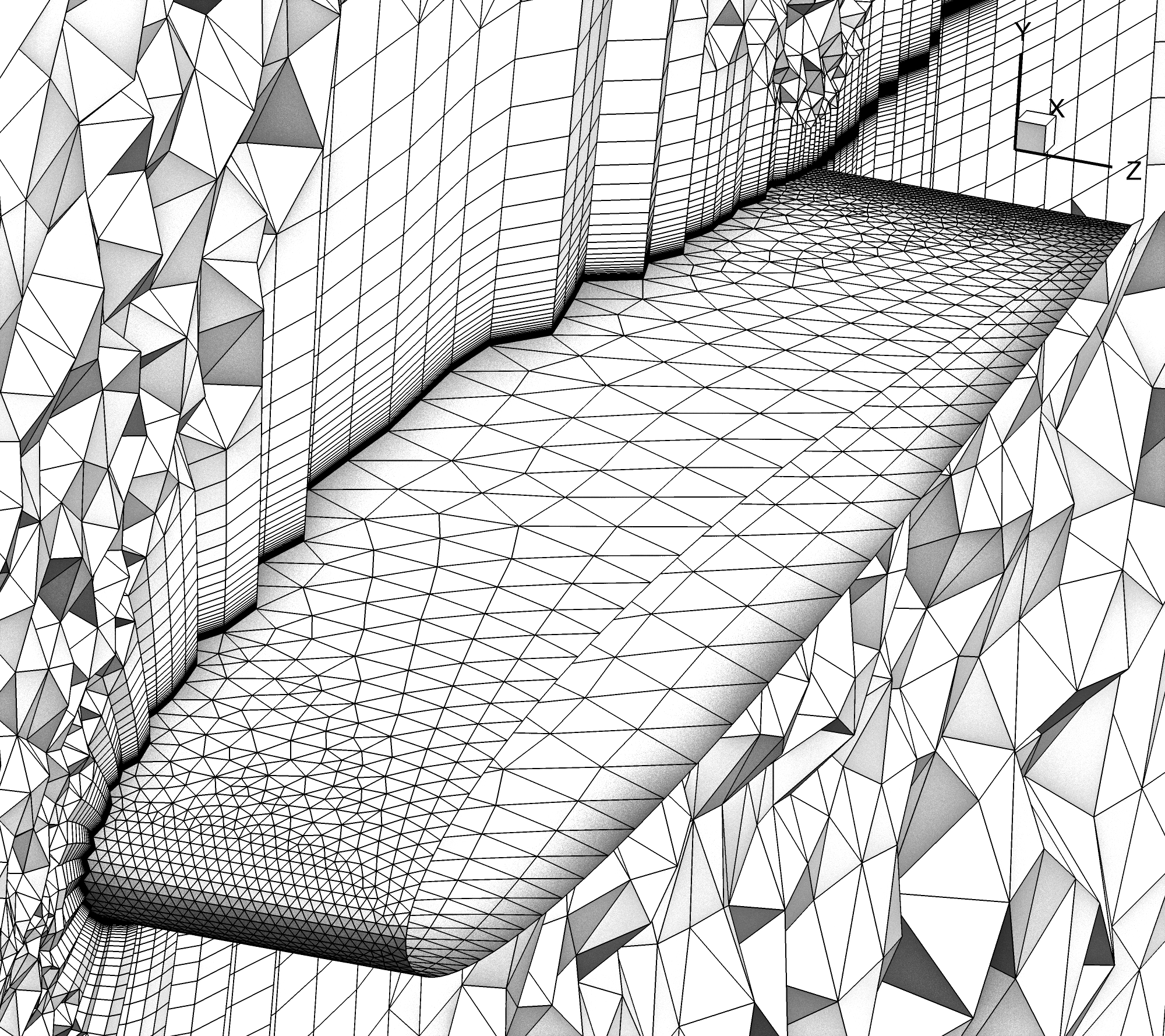}
\caption{Mesh cx4 for the flow around the NACA0012 body}\label{fig:meshNACA}
\end{figure}

On the body surface, the no-slip adiabatic boundary conditions are set. The periodic boundary conditions in $z$ are used, and on the rest of the boundary we keep the background flow 
$$
\bm{U} = (\bar{\rho}, \bar{u}, \bar{v}, \bar{w}, \bar{p}, \tilde{\nu})^T = \left(1, \cos\phi, \sin\phi, 0, \frac{1}{\gamma M^2}, \mu\right)^T.
$$
The angle of attack $\phi$ is 10 degrees, the Mach number is $M = 0.15$. The horde length is unit. The Reynolds number is $\mathrm{Re} = 6 \cdot 10^6$. The molecular viscosity $\mu$ is defined by the Sutherland law with $T_{\infty} = 300K$.

We present the numerical results for the drag and lift coefficients in Table~\ref{table:cdcl3d}. The reference data is obtained using the 2D computation on the 2D mesh and agrees with \cite{NACA0012}. The viscous part of the lift coefficient is negligible and not presented.

\begin{table}[t]
\caption{\label{table:cdcl3d}Drag and lift coefficients}
\begin{center}
\begin{tabular}{|c|c|r|r|r|r|}
\hline 
Coefficient & Scheme & cx4 & cx3 & cx2 & Ref. \\
\hline
\multirow{ 3}{*}{$C_{d,conv}$} & BBR3 & $7.43 \cdot 10^{-3}$ & $6.79 \cdot 10^{-3}$ & $6.42 \cdot 10^{-3}$ & \multirow{ 3}{*}{$6.2 \div 6.3 \cdot 10^{-3}$} \\
& EBR3-PL & $9.08 \cdot 10^{-3}$ & $7.39 \cdot 10^{-3}$ & $6.57 \cdot 10^{-3}$ & \\
& EBR5-PL & $6.44 \cdot 10^{-3}$ & $6.30 \cdot 10^{-3}$ & $6.15 \cdot 10^{-3}$ & \\
\hline
\multirow{ 3}{*}{$C_{d,visc}$} & BBR3 & $5.98 \cdot 10^{-3}$ & $6.11 \cdot 10^{-3}$ & $6.16 \cdot 10^{-3}$ & \multirow{ 3}{*}{$6.16 \div 6.18 \cdot 10^{-3}$} \\
& EBR3-PL & $5.88 \cdot 10^{-3}$ & $6.07 \cdot 10^{-3}$ & $6.17 \cdot 10^{-3}$ & \\
& EBR5-PL & $6.13 \cdot 10^{-3}$ & $6.17 \cdot 10^{-3}$ & $6.19 \cdot 10^{-3}$ & \\
\hline
\multirow{ 3}{*}{$C_{l}$} & BBR3 & $1.0798$ & $1.0848$ & $1.0872$ & \multirow{ 3}{*}{$1.087 \div 1.098$} \\
& EBR3-PL & $1.0808$ & $1.0897$ & $1.0923$ & \\
& EBR5-PL & $1.1037$ & $1.098$ & $1.094$ & \\
\hline
\end{tabular}
\end{center}
\end{table}

The surface mesh consists of triangles that are close to regular ones. This allows to assume that the solution error consists of two components. The first one is defined by the behavior of the scheme for linear equations on a regular-triangular mesh (near a planar body surface) extruded in the normal direction. The second component results from the effects that an actual mesh is non-uniform, the boundary is not planar, nonlinear effects, etc.

On regular prismatic meshes for linear equations, the EBR3-PL scheme is 3-exact and the EBR5-PL scheme is 5-exact. Table~\ref{table:cdcl3d} shows that the EBR5-PL scheme gives better results than the EBR3-PL scheme. This means that the first component of the solution error does matter and even prevails on coarse meshes.

On regular prismatic meshes for linear equations, the BBR3 scheme is only 2-exact while the EBR3-PL scheme is 3-exact. However, Table~\ref{table:cdcl3d} shows  that the accuracy of the BBR3 scheme is nearly the same as the accuracy of the EBR3-PL scheme (and even marginally better, but a cell-centered scheme on prismatic meshes has twice more degrees of freedom than a vertex-centered one). We believe that the clue to understand this effect is Proposition~\ref{th:av:BBR}, which states the zero mean error on the 3-order polynomials that give a steady solution of the governing equations.

\section{Conclusion}

A scheme is called supra-convergent when the numerical solution converges with a higher order than what is predicted by the truncation error analysis. In this paper, we studied the supra-convergence of finite-volume schemes for hyperbolic equations on unstructured meshes.

The key point of our study is the zero mean error condition. It means that for a $p$-exact scheme, the truncation error on $(p+1)$-th order polynomials should have zero mean. This condition is proved for the FV schemes with a polynomial reconstruction, 1-exact edge-based schemes, and the flux correction method. 

The zero mean error condition is necessary for the supra-convergence (at least, for the transport equation) but not sufficient. Also this condition is relatively easy to check. We believe that the zero mean error condition is a useful tool to predict the supra-convergence when developing numerical methods on unstructured meshes.

\appendix 
\section*{Appendix}
\renewcommand{\thesection}{\Alph{section}}

\section{A simple example}
\label{sect:simple_example}

Here we illustrate the supra-convergence effect on a simple example. The idea of this example was taken from \cite{Pascal2007}. Consider the Cauchy problem for the 1D transport equation $w_t + w_x = 0$, $w(0,x)=w_0(x)$. Let $x_j$, $j \in \mathbb{Z}$ be the mesh nodes (assume $x_{j+1} > x_j$ for each $j$). Denote
$$
x_{j+1/2} = (x_j + x_{j+1})/2, \quad h_{j+1/2} = x_{j+1}-x_j, \quad \newhbar_j = x_{j+1/2}-x_{j-1/2}.
$$
Consider the following scheme on the dual cells $(x_{j-1/2}, x_{j+1/2})$:
\begin{equation}
\frac{u_j^{n+1}-u_j^n}{\tau} + \frac{u_{j}^n-u_{j-1}^n}{\newhbar_j} = 0, \quad u_j^0 = w_0(x_j), \quad j \in \mathbb{Z}, \quad n \in \mathbb{N} \cup \{0\}.
\label{eq_basic}
\end{equation}

The truncation error of \eqref{eq_basic} at node (dual cell) $j$ is
$$
\epsilon_j^n = \frac{\tau}{2} w_{tt}(\theta_1,x_j) + \frac{h_{j+1/2}-h_{j-1/2}}{\newhbar_j} w_x(t_n,x_j) + \frac{h_{j-1/2}^2}{2\newhbar_j} w_{xx}(t_n,\theta_2)
$$
with $t_n < \theta_1 < t_{n+1}$, $x_{j-1} < \theta_2 < x_j$. Obviously, on a non-uniform mesh the scheme is not 1-exact $(p=0)$.

However, it is well-known that this scheme has the first order of accuracy. This can be easily proved basing on the 1-exactness in the sense of the map taking each $f \in C(\mathbb{R})$ to the mesh function with the components $f(x_{j+1/2})$. 

\begin{proposition}
Let $w_0 \in C^2(\mathbb{R}) \cap W^{2,\infty}(\mathbb{R})$, $w(t,x)=w_0(x-t)$, and $u_j^n$ be the solution of \eqref{eq_basic}. Then for $\tau \le \min_j \newhbar_j$ there holds
$$
\sup\limits_{j \in \mathbb{Z}} |u_j^n - w(t_n,x_j)| \le t_n h_{\max} \|w_0''\|_{\infty} + h_{\max} \|w_0'\|_{\infty}
$$
where $t_n = n\tau$.
\end{proposition}
\begin{proof}
For a bounded sequence $f_j$, $j \in \mathbb{Z}$, denote $\|f\|_{\infty} = \sup_j |f_j|$. Put $\varepsilon_j^n = u_j^n - w(t_n,x_{j+1/2})$. From \eqref{eq_basic} we have
\begin{equation}
\frac{1}{\tau}(\varepsilon_j^{n+1} - \varepsilon_j^n) + \frac{1}{\newhbar_j}(\varepsilon_j^n-\varepsilon_{j-1}^n) = - \tilde{\epsilon}_j^n, \quad \varepsilon_j^0 = w(0,x_{j}) - w(0,x_{j+1/2}),
\label{eq_aux_1111}
\end{equation}
$$
\tilde{\epsilon}_j^n = \frac{1}{\tau}(w(t_{n+1},x_{j+1/2})-w(t_n,x_{j+1/2})) + \frac{1}{\newhbar_j}(w(t_n,x_{j+1/2}) - w(t_n,x_{j-1/2})). 
$$
Using the Taylor formula at $(t_n, x_{j+1/2})$ we get
$$
\tilde{\epsilon}_j^n = \frac{\tau}{2} w_{tt}(\theta_t,x_{j+1/2}) + \frac{\newhbar_j}{2} w_{xx}(t_n,\theta_x)
$$
where $t_n \le \theta_t \le t_{n+1}$ and $x_{j-1/2} \le \theta_x \le x_{j+1/2}$. From here,
$$
\|\varepsilon^0\|_{\infty} \le \frac{1}{2}h_{\max} \|w_0'\|_{\infty}, \quad \|\tilde{\epsilon}^n\|_{\infty} \le h_{\max} \|w_0''\|_{\infty}.
$$
From \eqref{eq_aux_1111}, under the assumption $\tau \le \min_j \newhbar_j$ we obtain
$$
\|\varepsilon^{n+1}\|_{\infty} \le \|\varepsilon^{n}\|_{\infty} + \tau \|\tilde{\epsilon}^{n}\|_{\infty}.
$$
Replacing $n$ by $n-1$ and continuing the chain, we get
$$
\|\varepsilon^{n}\|_{\infty} \le \|\varepsilon^{0}\|_{\infty} + \tau \sum\limits_{k=0}^{n-1} \|\tilde{\epsilon}^{k}\|_{\infty} \le \frac{1}{2}h_{\max} \|w_0'\|_{\infty} + t_n h_{\max} \|w_0''\|_{\infty}.
$$
It remains to use the triangle inequality
$$
\sup\limits_{j \in \mathbb{Z}} |u_j^n - w(t_n,x_j)| \le 
\sup\limits_{j \in \mathbb{Z}} |u_j^n - w(t_n,x_{j+1/2})| + 
\sup\limits_{j \in \mathbb{Z}} |w(t_n,x_j) - w(t_n,x_{j+1/2})|
\le
$$
$$
\le \|\varepsilon^{n}\|_{\infty}  + \frac{1}{2}h_{\max} \|w_0'\|_{\infty},
$$
which concludes the proof.
\end{proof}

\section{Barycentric control volumes}
\label{sect:barycentric}

Recall the definition of the barycentric control volumes. A point $\bm{r} \in \mathbb{R}^d$ lays inside a simplex $e \in \mathcal{E}$ or on its boundary. Then 
$$
\bm{r} = \sum\limits_{j \in e} \alpha_j \bm{r}_j, \quad \alpha_j \ge 0, \quad \sum\limits_{j \in e} \alpha_j = 1.
$$
The point $\bm{r}$ belongs to the control volume $K_j$ of any vertex $j$ with the largest coefficient in this combination:
$$
K_j = \bigcup\limits_{e \in \mathcal{E}} \left\{ \sum\limits_{l \in e} \alpha_l \bm{r}_l,\ \alpha_l \in [0,1],\ \sum\limits_{l \in e}\alpha_l = 1,\ \alpha_j = \max_{l\in e} \alpha_l\right\}.
$$

The vector $\bm{n}_{jk}$ defined by \eqref{eq_def_njk} may be represented in the form
$$
\bm{n}_{jk} = \sum\limits_{e \ni j, k} \bm{n}_{jk, e}, \quad
\bm{n}_{jk,e} = \int\limits_{\D K_j \cap \D K_k \cap e} \bm{n} dS
$$
where $\bm{n}$ is the unit normal oriented inside $K_k$. For each $e \in \mathcal{E}$, $j, k \in e$ there holds
\begin{equation}
\D K_j \cap \D K_k \cap e = \left\{ \sum\limits_{l \in e} \alpha_l \bm{r}_l,\ \alpha_l \in [0,1],\ \sum\limits_{l \in e}\alpha_l = 1,\ \alpha_j = \alpha_k = \max_{l\in e} \alpha_l\right\}.
\label{eq_def_kjkke}
\end{equation}
This set belongs to the $(d-1)$-simplex with the vertices $(\bm{r}_j + \bm{r}_k)/2$ and $\bm{r}_l$, $l \in e \setminus \{j,k\}$.

Let $e$ be a simplex and $\Gamma(e)$ be the set of its faces. For $j \in e$, let $\gamma(e,j) \in \Gamma(e)$ be the face of $e$ that does not contain node $j$ and
$$
\bm{n}_{\gamma(e,j)} = \int\limits_{\gamma(e,j)} \bm{n} ds,
$$
where $\bm{n}$ is the unit normal pointing outside $e$. 

\begin{lemma}\label{th:B1}
For each $e \in \mathcal{E}$ and each $j, k \in e$ there holds
\begin{equation}
\bm{n}_{jk,e} = c (\bm{n}_{\gamma(e,j)} - \bm{n}_{\gamma(e,k)})
\label{eq_B1}
\end{equation}
where $c>0$ depends on $d$ only.
\end{lemma}
\begin{proof}

For $\alpha \in [0,1]$, let $E(\alpha)$ be the $(d-1)$-simplex with the vertices at $\bm{r}_l$, $l \in e \setminus \{j,k\}$, and  $(1-\alpha)\bm{r}_j + \alpha\B{r}_k$. Let $\bm{S}(\alpha)$ be its oriented area such that $\bm{S}(\alpha) \cdot (\bm{r}_k - \bm{r}_j) \ge 0$. By construction, $\bm{S}(1) = \bm{n}_{\gamma(e,j)}$, $\bm{S}(0) = - \bm{n}_{\gamma(e,k)}$. Since $\bm{S}(\alpha)$ is a linear function of $\bm{\alpha}$, then $\bm{n}_{\gamma(e,j)} - \bm{n}_{\gamma(e,k)} = 2\B{S}(1/2)$. 

There holds
$$
E(1/2) = \left\{ \sum\limits_{l \in e} \alpha_l \bm{r}_l,\ \alpha_l \in [0,1],\ \sum\limits_{l \in e}\alpha_l = 1,\ \alpha_j = \alpha_k\right\}.
$$
Comparing this with \eqref{eq_def_kjkke} we see that $\D K_j \cap \D K_k \cap e \subset E(1/2)$ and $|\D K_j \cap \D K_k \cap e| = 2c|\bm{S}_{1/2}|$ with a constant $c$ that depends on $d$ only.
\end{proof}

\begin{lemma}\label{th:B2}
For each $e \in \mathcal{E}$ and each $j, k \in e$ there holds
\begin{equation}
\bm{n}_{jk,e} = \frac{1}{d(d+1)} (\bm{n}_{\gamma(e,j)} - \bm{n}_{\gamma(e,k)})
\label{eq_B1a}
\end{equation}
\end{lemma}
\begin{proof}

For $l, m \in e$ let $e_{l,m}$ be the pyramid with the base $\D K_l \cap \D K_m \cap e$ and the apex $l$. These pyramids have intersections of zero $d$-dimensional  measure and 
$$
e = \bigcup\limits_{l, m \in e; l \ne m} e_{l,m}.
$$
Then
$$
|e| = \sum\limits_{l, m \in e; l \ne m} |e_{l,m}| = 
\sum\limits_{l, m \in e; l \ne m} \frac{1}{d} \frac{\bm{r}_m - \bm{r}_l}{2} \cdot \bm{n}_{lm,e} = 
$$
$$
=
c\sum\limits_{l, m \in e; l \ne m} \frac{1}{2d} (\bm{r}_m - \bm{r}_l) \cdot (\bm{n}_{\gamma(e,l)} - \bm{n}_{\gamma(e,m)}).
$$
The last identity is by~\ref{eq_B1}.
Since $(\bm{r}_m - \bm{r}_l) \cdot \bm{n}_{\gamma(e,l)} = - (\bm{r}_m - \bm{r}_l) \cdot \bm{n}_{\gamma(e,m)} = d|e|$, we obtain
$$
|e| = c\sum\limits_{l, m \in e; l \ne m} |e| = cd(d+1)|e|.
$$
From here, $c = 1/(d(d+1))$, and \eqref{eq_B1a} follows from \eqref{eq_B1}.
\end{proof}

Let $\{\phi_j, j \in \mathcal{N}\}$ be the standard P1-Galerkin basis, i. e. $\phi_j$ are linear within each simplex, $\phi_j(\bm{r}_j) = 1$, and $\phi_j(\bm{r}_k) = 0$ for $j \ne k$. 

\begin{lemma}
For each $e \in \mathcal{E}$ and each $j, k \in e$ there holds
$$
\bm{n}_{jk, e} = \int\limits_e (\phi_k \nabla \phi_j - \phi_j \nabla \phi_k) d\B{r}.
$$
\end{lemma}
\begin{proof}
For $\bm{r} \in e$, we have $\nabla \phi_j = \bm{n}_{\gamma(e,j)}/(d|e|)$. Also $\int_e \phi_k d\B{r} = |e|/(d+1)$. It remains to use \eqref{eq_B1a}.
\end{proof}

\begin{corollary}
For $j \in \mathcal{N}$ and $k \in N(j)$ there holds
$$
\bm{n}_{jk} = \int\limits_{\mathbb{R}^d} (\phi_k \nabla \phi_j - \phi_j \nabla \phi_k) d\B{r}.
$$
\end{corollary}

\section*{Acknowledgements}

The computations were carried out using the equipment of Shared Resource Center of KIAM RAS (http://ckp.kiam.ru). The authors thankfully acknowledge this institution.


\end{document}